\documentclass[mainlanguage=english,fncychap=none]{amsart}

\usepackage[utf8]{inputenc}
\usepackage[english]{babel}
\usepackage[T1]{fontenc}
\usepackage{amsfonts}
\usepackage{amsmath}
\usepackage{amssymb}
\usepackage{amsthm}
\usepackage{mathtools}
\usepackage{relsize}

\usepackage{tikz-cd}
\usepackage[a4paper, total={6.6in, 8.6in}]{geometry}
\usepackage{graphicx}
\usepackage[hypcap=false]{caption} 
\usepackage{subcaption}
\usepackage{cite}
\usepackage[hidelinks]{hyperref}
\usepackage{cleveref}\usepackage{mathtools}

\newtheorem{thm}{Theorem}[section]
\newtheorem{innercustomthm}{Theorem}
\newenvironment{customthm}[1]
  {\renewcommand\theinnercustomthm{#1}\innercustomthm}
  {\endinnercustomthm}
\newtheorem{prop}[thm]{Proposition}
\newtheorem{corollary}[thm]{Corollary}
\newtheorem{lemma}[thm]{Lemma}
\newtheorem{defn}[thm]{Definition}
\newtheorem{rqe}[thm]{Remark}

\theoremstyle{definition}
\newtheorem*{ack}{Acknowledgments}

\newcommand{\Z}{\mathbb{Z}}
\newcommand{\Sr}{\mathbb{S}}

\newcommand{\M}{\mathcal{M}}
\newcommand{\I}{\mathcal{I}}
\newcommand{\K}{\mathcal{K}}

\title{Triple-cup product forms of $3$-manifolds and Heegaard diagrams}
\author{Maya KAYALI}
\address{Université Bourgogne Europe, CNRS, IMB (UMR 5584), 21000 Dijon, France}
\email{maya.kayali@ube.fr}
\date{}

\begin{document}

\maketitle

\begin{abstract}
The triple-cup product form $\mu$ is a classical invariant of $3$-manifolds, determining the cohomology ring up to torsion. Given a closed, connected, oriented $3$-manifold $M$, we describe an explicit formula for computing $\mu$ from a Heegaard diagram of $M$. Then, we show that the triple-cup product form $\mu$ can be recovered as a reduction of Turaev's homotopy intersection form $\eta$ of the Heegaard surface. 
\end{abstract}

\tableofcontents

\section{Introduction}

Let $M$ be a closed, connected, oriented $3$-manifold. A Heegaard splitting of $M$ is a decomposition of $M$ into two handlebodies $H_g$ glued via a diffeomorphism $f:~\partial~H_g~\to~\partial~(-H_g)$ along their boundaries 
\[
M = H_g \cup_f (-H_g).
\]
Since the diffeomorphism type of $M$ only depends on the gluing map $f$ up to isotopy, it is sufficient to know the images of the meridians of the two handlebodies in order to describe $M$. A Heegaard splitting can thus be represented by a diagram. We label the two handlebodies $H_\alpha$ and $H_\beta$, such that $M = H_\alpha \cup_f H_\beta$, and we denote $\Sigma~:=~ \partial H_\alpha = -\partial H_\beta$ the Heegaard surface in $M$. Then, a Heegaard diagram is given by a triplet $(\Sigma, \alpha, \beta)$, where $\alpha$ and $\beta$ are two families of $g$ simple closed curves, denoted $\alpha_1,\ldots,\alpha_g$ and $\beta_1,\ldots,\beta_g$, corresponding to the images of meridians of $H_\alpha$ and $H_\beta$, respectively, in $\Sigma$ (see for example \cite{Saveliev}).

Let $(\Sigma,\alpha,\beta)$ be a Heegaard diagram of $M$. For $\varepsilon\in\{\alpha,\beta\}$, the inclusion of the Heegaard surface $\Sigma$ into $H_\varepsilon$ induces a map in homology (with ordinary coefficients in $\Z$):
\[
\iota_\varepsilon : H_1(\Sigma_g) \to H_1(H_\varepsilon).
\]
We set
\[
L_\varepsilon := \ker\bigl(H_1(\Sigma)\xrightarrow{\iota_\varepsilon} H_1(H_\varepsilon)\bigr).
\]
The group $H_1(\Sigma)$ equipped with the intersection form $\omega : H_1(\Sigma)\times H_1(\Sigma)\to \Z$ is a symplectic space. The groups $L_\varepsilon$ are Lagrangian subgroups of ($H_1(\Sigma),\omega)$ generated by the curves $\varepsilon$. 

The homology of $M$ can be fully described in terms of the subgroups $L_\alpha$ and $L_\beta$ using the long exact sequences associated with the pairs $(H_\alpha,\Sigma)$, $(H_\beta,\Sigma)$ and $(M,\Sigma)$:
\begin{equation}\label{eq:iso_H1_H2}
\begin{split}
H_1(M) &\cong \frac{H_1(\Sigma)}{L_\alpha + L_\beta}, \\
H_2(M) &\cong L_\alpha \cap L_\beta.
\end{split}
\end{equation}

We now focus more specifically on the cohomology ring of $M$. Using the Poincaré duality and the universal coefficient theorem, we can recover up to torsion all the cohomology groups of $M$ from $H^1(M)$. Moreover, the ring structure of $H^*(M)$ is entirely determined by the triple-cup product form
\[
\begin{array}{cccc}
\mu : & H^1(M)\times H^1(M)\times H^1(M) & \longrightarrow & \Z \\
      & (x,y,z) & \mapsto & \langle x\smile y\smile z, [M] \rangle.
\end{array}
\]
Indeed, for $u,v\in H^1(M)$, the map $\mu(u,v,\cdot)\in \mathrm{Hom}(H^1(M),\Z)$ corresponds by Poincaré duality to the class~$u\smile v$ in $H^2(M)/\mbox{Tors}(H^2(M))$. 

It is well-known how to compute $\mu$ from a surgery presentation of $M$ using Milnor's invariants of the surgery link (see for instance \cite{Turaev84}). However, although $\mu$ is a classical topological invariant, we are not aware of any explicit computation in the literature from a Heegaard diagram of $M$. Such formulas in terms of Heegaard diagrams do exist for other classical invariants such as the linking form \cite{ConwayFrieldHerrmann, BirmanJohnsonPutman}. But it is important to note that the computation of the linking form depends only on the homomorphism map induced on $H_1(\Sigma)$ by the gluing map of the Heegaard splitting, which is not the case for the triple-cup product form. 

In this paper, we describe a formula for computing the triple-cup product form of a $3$-manifold from a Heegaard diagram in terms of the identification (\ref{eq:iso_H1_H2}). More precisely, using the Poincaré duality, the cup product of two cocycles in $H^1(M)$ can be regarded as a $1$-cycle in $H_1(\Sigma)$ mod $(L_\alpha + L_\beta)$. Our first result, Theorem~$\mathrm{\ref{thm:A}}$ below, describes how to build such $1$-cycles using arcs sitting in $\alpha$-curves and $\beta$-curves, and coefficients determined from the intersection pattern of these curves. The explicit definition of these arcs and these coefficients is too long to be stated here, but is detailed in Section \ref{description}. 

\begin{customthm}{A}[see Theorem \ref{1-cycle}]
\label{thm:A}
Let $x$ and $x'\in H^1(M)$. Let $c_\alpha$ be a disjoint union of $\alpha$-curves and $c_\beta'$ be a disjoint union of $\beta$-curves, such that $[c_\alpha]$ and $[c_\beta']$ are homology classes in $L_\alpha \cap L_\beta$ associated by Poincaré duality to $x$ and $x'$, respectively. Then the cup product $x\smile x'\in H^2(M)\cong H_1(M)$ is represented by a $1$-cycle in $\Sigma$ of the following form:
\[
\sum_{k,i} r_i^k a_i^k +\sum_{k,i} s_i^k b_i^k
\]
where $a_i^k$ (respectively $b_i^k$) are arcs in $c_\alpha$ (respectively in $c_\beta'$), and $r_i^k$, $s_i^k$ are coefficients in $\Z$. 
\end{customthm}

We refer to Section \ref{description} for an explicit definition of the arcs $a_i^k$, $b_i^k$ and the coefficients $r_i^k$, $s_i^k$. Then, the triple cup product $x\smile x'\smile x''$ for $x,x,'x''\in H_1(M)$ is obtained as the algebraic number of intersections between the $1$-cycle of Theorem $\mathrm{\ref{thm:A}}$ and some $\alpha$-curves representing $x''$ (see Corollary \ref{triple_cup_product}). Section~\ref{proof} is devoted to the proof of the formula. 

Besides, Turaev introduced a homotopy intersection form $\eta$ of surfaces in \cite{Turaev79}. In Section~\ref{Turaev}, we deduce from Theorem $\mathrm{\ref{thm:A}}$ that the cup product in $M$ can be recovered as a certain reduction $\overline{\eta}$ of Turaev’s homotopy intersection form of the Heegaard surface $\Sigma$. By denoting $\pi:=\pi_1(\Sigma\setminus\mathrm{int}(D))$ where $D$ is a small disk in $\Sigma$, $A:=\mathrm{ker}(\pi\to\pi_1(H_\alpha))$ and $B:=\mathrm{ker}(\pi\to\pi_1(H_\beta))$, we have the following: 

\begin{customthm}{B}[see Theorem \ref{thm:eta_varphi}]
\label{thm:B}

The following diagram is commutative 

\begin{center}
\begin{tikzcd}
\displaystyle{\frac{A\cap B[\pi,\pi]}{A\cap [\pi,\pi]}} \arrow[d,shift left=0.5cm, "\cong"] \hspace{-1cm}& \times &\hspace{-1cm}  \displaystyle{\frac{B\cap A[\pi,\pi]}{A\cap [\pi,\pi]}}\arrow[d,"\cong",shift left=-0.5cm] \arrow[rr,"\overline{\eta}"]& & \displaystyle{\frac{H_1(\Sigma)}{L_\alpha + L_\beta}}\arrow[dd,"\cong"] \\
L_\alpha \cap L_\beta \arrow[d,"\cong",shift left=0.5cm] \hspace{-1cm}&&\hspace{-1cm} L_\alpha \cap L_\beta\arrow[d,"\cong",shift left=-0.5cm]  &&\\
H^1(M) \hspace{-1cm}& \times &\hspace{-1cm} H^1(M) \arrow[rr,"\smile"] & & H^2(M).
\end{tikzcd}
\end{center}

\end{customthm}

In particular, we recover from Theorem $\mathrm{\ref{thm:B}}$ some classical properties of the triple-cup product, such as skew-symmetry and invariance under twists along surfaces by elements of the Johnson kernel (see Subsection~\ref{varphi_properties}). Finally, Section~\ref{section:examples} presents two examples of computation of cohomology rings from Heegaard diagrams.

\begin{ack}
The author would like to thank her advisors Renaud Detcherry and Gwénaël Massuyeau for many helpful discussions and their careful proofreading. She also thanks the anonymous referee for their careful reading and insightful comments and Takefumi Nosaka for communicating another more algebraic way of calculating the form of triple cups from Heegaard splittings. This work has been achieved at the IMB (Dijon) which receives support from of the EIPHI Graduate school (contract "ANR-17-EURE-0002"). 
\end{ack}

\section{Description of the formula}
\label{description}

Let $M$ be a connected closed oriented $3$-manifold and let $(\Sigma,\alpha,\beta)$ be a Heegaard diagram of $M$. We fix an orientation for the curves $\alpha$ and $\beta$ of the diagram. Using the Poincaré duality and the Mayer-Vietoris exact sequence, we can describe the elements of $ H^1(M) $ by pairs of homologous curves in the Heegaard surface:
\[
H^1(M) \cong H_2(M) \cong L_{\alpha} \cap L_{\beta} \cong \left\{ \left( \left[ c_{\alpha} \right], \left[ c_{\beta} \right] \right) \in L_{\alpha} \times L_{\beta} \mid \left[ c_{\alpha} \right] = \left[ c_{\beta} \right] \in H_1(\Sigma)\right\}.
\]

An element $\left[ c_{\alpha} \right] \in L_{\alpha}$ can be written as $\left[ c_{\alpha} \right] = \sum_{i=1}^{g} \lambda_i \left[ \alpha_i \right]$. It is then represented by the disjoint union of~$|\lambda_i|$ parallel copies of $\alpha_i$ (the sign of $\lambda_i$ indicates whether the orientation corresponds to the one fixed above).
Thus, we can describe the elements of~$ H^1(M) $ by pairs of multicurves $(c_{\alpha}, c_{\beta})$, homologous in $ H_1(\Sigma) $, obtained as a union of parallel copies of $\alpha$ and $\beta$ curves. 

We start by calculating the cup product of two elements $ x, x' \in H^1(M) $. We choose~$(c_{\alpha}, c_{\beta})$ a pair of oriented homologous multicurves representing~$x$, and $(c_{\alpha}', c_{\beta}')$ a pair of oriented homologous multicurves representing~$x'$. We denote by $\varphi$ the map defined by the commutative diagram:

\vspace{0.5cm}

\begin{center}
\begin{tikzcd}
\begin{matrix}H^1(M)\times H^1(M)\\ (x,x')\end{matrix} \arrow[r, "\smile"] \arrow[d,"\cong"] & \begin{matrix}H^2(M)\\x\smile x'\end{matrix} \arrow[d, "\cong"] \\ 
\begin{matrix}(L_\alpha \cap L_\beta)\times(L_\alpha \cap L_\beta)\\ ([c_\alpha],[c'_\beta])\end{matrix} \arrow[r, "\varphi"] & \begin{matrix}H_1(M)\\\varphi(c_\alpha,c'_\beta)\end{matrix}
\end{tikzcd}
\end{center}

\vspace{0.5cm}

This section provides a formula for calculating the 1-cycle $\varphi(c_{\alpha}, c_{\beta}') \in H_1(M)$, which can be decomposed into two parts: a 1-chain expressed in terms of arcs of $c_{\alpha}$ and a 1-chain expressed in terms of arcs of $c_{\beta}'$. 

We begin with the part of $\varphi(c_{\alpha}, c_{\beta}')$ contained in $c_{\alpha}$. We denote by $c_{\alpha}^1, \ldots, c_{\alpha}^N$ the~$N$ connected components of $c_{\alpha}$. The algebraic intersection of $c_{\beta}'$ with each $c_{\alpha}^k$ is trivial because $[c_{\beta}'] \in L_{\alpha}$. Therefore, there exists $n_k \in \mathbb{N}$ such that the component $c_{\alpha}^k$ intersects $c_{\beta}'$ $2n_k$ times with $n_k$ positive intersections and $n_k$ negative intersections (the intersections may involve several components of $c_{\beta}'$).

For $k \in \{1, \ldots, N\}$, we number the arcs of $c_{\alpha}^k \setminus (c_{\alpha}^k \cap c_{\beta}')$ from 1 to $2n_k$ choosing a base point on $c^k_\alpha$ and following the given orientation of $c_{\alpha}^k$. Taking indices modulo~$2n_k$, we set $\varepsilon_i(c'_\beta,c^k_\alpha) = 1$ if $c_{\beta}'$ intersects $c_{\alpha}^k$ positively between arcs $i$ and $i+1$, and $\varepsilon_i(c'_\beta,c^k_\alpha) = -1$ if the intersection is negative. We define a sequence of integers $(r_i^k)_{1 \leq i \leq 2n_k}$ by
\[ r_i^k = \sum_{j=1}^{i-1} \varepsilon_j(c_\beta',c^k_\alpha).\]

See Figure $\ref{fig:notations}$ for an example.

\begin{rqe} 
If the component $c_{\alpha}^k$ does not intersect $c_{\beta}'$, it will not contribute to~$\varphi(c_{\alpha}, c_{\beta}')$ and we do not take it into account.
\end{rqe}

\begin{figure}[t]
\includegraphics[scale=0.5]{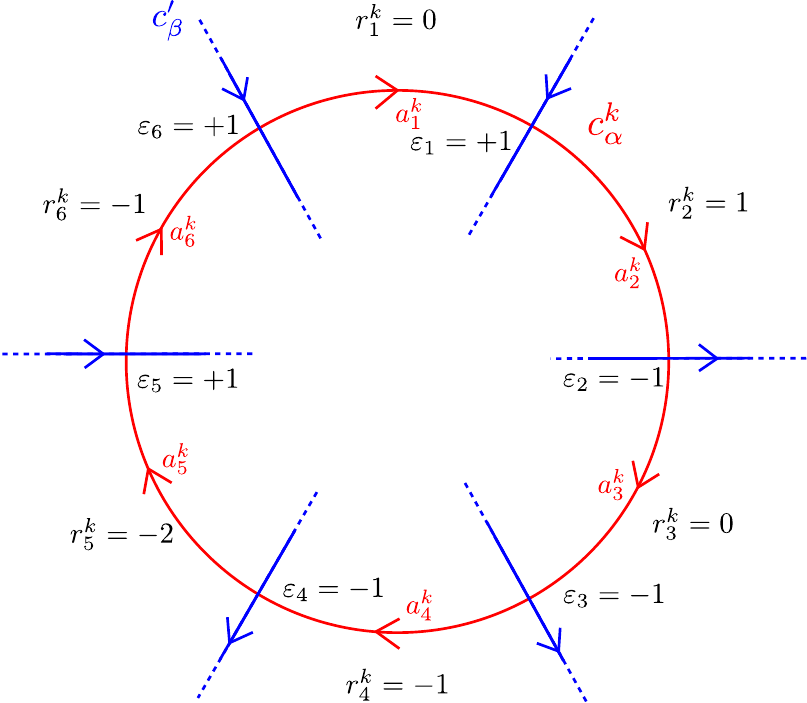}
{\caption{Example of computation of coefficients $r_i^k$. The basepoint is chosen on the arc $a_1^k$.}
\label{fig:notations}}
\end{figure}

We then proceed similarly to compute the part of $\varphi(c_{\alpha}, c_{\beta}')$ coming from $c_{\beta}'$. As before, we denote by~$c_{\beta}'^1, \ldots, c_{\beta}'^M$ the $M$ components of $c_{\beta}'$. For $k \in \{1, \ldots, M\}$, the algebraic intersection of a component~$c_{\beta}'^k$ with $c_{\alpha}$ is trivial because $[c_{\alpha}] \in L_{\beta}$. Therefore, there exists $m_k \in \mathbb{N}$ such that $c_{\beta}'^k$ intersects $c_{\alpha}$ $2m_k$ times, with~$m_k$ positive intersections and $m_k$ negative intersections. We label the arcs of $c_{\beta}'^k \setminus (c_{\beta}'^k \cap c_{\alpha})$ from 1 to $2m_k$ choosing a base point and following the given orientation of $c_{\beta}'^k$, then we introduce $(s_i^k)_{1 \leq i \leq 2m_k}$ a sequence of relative integers defined by:
\[s_i^k = \sum_{j=1}^{i-1} \varepsilon_j(c_\alpha,c'^k_\beta).\]
where $\varepsilon_i(c_\alpha,c'^k_\beta) \in \{1, -1\}$ is the sign of the intersection point of $c_{\alpha}$ with $c_{\beta}'^k$ between the arcs $i$ and $i+1$. 

Now that the coefficients $r_i^k$ and $s_i^k$ have been explicitly defined, we can state a precise version of Theorem $\mathrm{\ref{thm:A}}$ presented in the introduction.  

\begin{thm}
\label{1-cycle}
Denoting $a_i^k$ the closure of the $i$-th arc in $c_{\alpha}^k$ and $b_i^k$ the closure of the $i$-th arc in $c_{\beta}'^k$, $\varphi(c_{\alpha}, c_{\beta}')$ is represented by the following 1-cycle in $H_1(M)$:
\[
\sum_{k=1}^N \sum_{i=1}^{2n_k} r_i^k a_i^k + \sum_{k=1}^M \sum_{i=1}^{2m_k} s_i^k b_i^k.
\]
\end{thm}
A proof of Theorem $\ref{1-cycle}$ is given in Section \ref{proof}, as a direct corollary of Lemmas \ref{lemma1} and~\ref{lemma2}. 

To compute the triple-cup product, we introduce a third element $x'' \in H^1(M)$, to which we associate a pair of multicurves $(c_{\alpha}'', c_{\beta}'')$. We have the following result:

\begin{corollary}
\label{triple_cup_product}
The triple-cup product $\mu(x, x', x'')$ is given by the algebraic intersection~$\omega([c_{\alpha}''],\varphi(c_\alpha,c_{\beta}'))$. 
\end{corollary}

We give a proof of Corollary $\ref{triple_cup_product}$ in the following section.

\section{Proof of the triple-cup product formula}
\label{proof}

To compute the cup product of two elements $x, x' \in H^1(M)$, we use the formula
\begin{equation*}
\label{eq:ScapS'}
x \smile x' = P([S' \cap S])
\end{equation*}
where $P$ denotes the Poincaré isomorphism and $S, S'$are surfaces representing $P(x)$ and $P(x')$ whose intersection is transverse in $M$. 

\begin{rqe}
Here, we follow the sign conventions of \cite{Bredon} to orient $S'\cap S$. For~$p~\in~S'\cap S$, take $\vec{\tau} \in T_p(S'\cap S)$, then complete it with $\vec{n}\in T_pS$ so that $(\vec{\tau},\vec{n})$ is a frame of $T_pS$ consistent with the orientation of $S$; then, complete it with a vector $\vec{n}'\in T_pS'$ so that $(\vec{\tau}, \vec{n}')$ is a frame of $T_pS'$ consistent with the orientation of $S'$; if $(\vec{\tau},\vec{n},\vec{n}')$ is consistent with the orientation of $T_pM$, then the orientation of $S'\cap S$ is given by $\vec{\tau}$; if not, the orientation of $S'\cap S$ is given by $-\vec{\tau}$. 
\end{rqe}

Let us describe how to construct such surfaces $S$ and $S'$ in the manifold $M$. As in the previous section, $(c_{\alpha}, c_{\beta})$ denotes a pair of multicurves representing the cohomology class $x$ and $(c_{\alpha}', c_{\beta}')$ a pair of multicurves representing the cohomology class $x'$. In particular, we have:
\[
\begin{aligned}
& [c_{\alpha}] = [c_{\beta}] \\
& [c_{\alpha}'] = [c_{\beta}'].
\end{aligned}
\]
Using the Heegaard splitting of $M$, we can describe $M$ as the union of two handlebodies $H_{\alpha}$ and $H_{\beta}$, glued along a thickening of the central surface:
\[
M = H_{\alpha} \cup_{\Sigma \times \{0\}} (\Sigma \times [0,1]) \cup_{\Sigma \times \{1\}} H_{\beta}
\]

We denote by $D_{\alpha}$ and $D_{\beta}$ the disjoint unions of meridian disks bounded by $c_{\alpha}$ and $c_{\beta}$ in $H_{\alpha}$ and $H_{\beta}$. Then the surface $S$ can be constructed as follows:
\[
S = D_{\alpha} \cup (c_{\alpha} \times [0, \tfrac{2}{3}]) \cup T \cup D_{\beta}
\]
where $T$ is a surface contained in $\Sigma \times [\frac{2}{3}, 1]$ satisfying $\partial T = c_{\beta}- c_{\alpha}$. Similarly, denoting by $D_{\alpha}', D_{\beta}'$ the disjoint unions of meridian disks bounded by $c_{\alpha}'$ and $c_{\beta}'$ in $H_{\alpha}$ and $H_{\beta}$, we can decompose $S'$ as follows:
\[
S' = D_{\alpha}' \cup T' \cup (c_{\beta}' \times [\frac{1}{3}, 1]) \cup D_{\beta}'
\]
where $T'$ is a surface contained in $\Sigma \times [0, \frac{1}{3}]$ satisfying $\partial T' = c_{\beta}'- c_\alpha'$ (see Figure \ref{surfaces}).

\vspace{0.5cm}

\begin{figure}[t]
\includegraphics[scale=0.27]{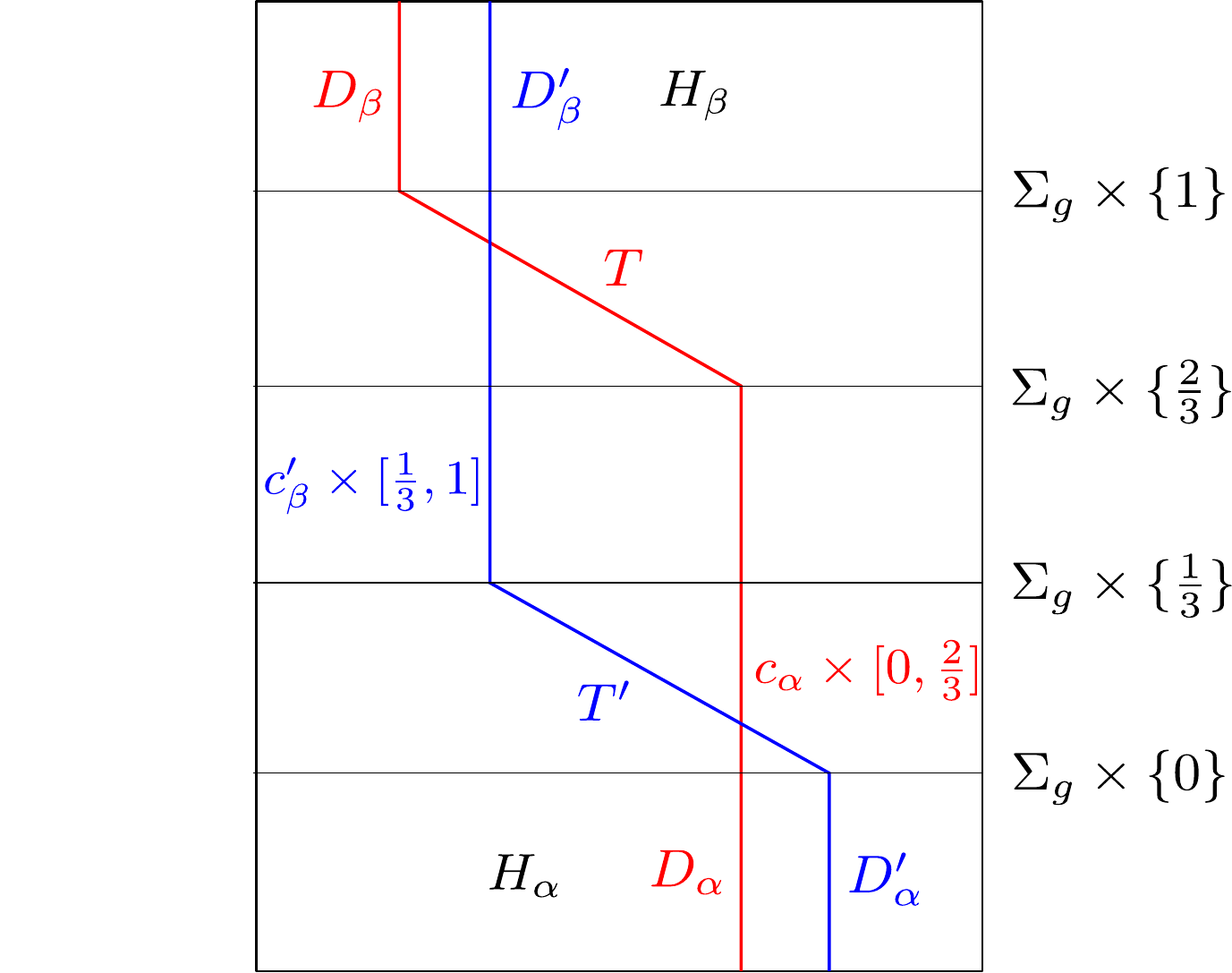}
{\caption{Schematic representation of the surfaces $S$ (in red) and $S'$ (in blue).}
\label{surfaces}}
\end{figure}

Then we can express the intersection between the surfaces $S$ and $S'$:

\begin{equation}\label{intersection_transverse}
S'\cap S = \Bigl(T'\cap (c_\alpha\times[0,\tfrac{1}{3}]) \Bigr) \cup \Bigl((c_\alpha\cap c'_\beta)\times [\tfrac{1}{3},\tfrac{2}{3}]\Bigl) \cup \Bigl(  (c'_\beta\times[\tfrac{2}{3},1])\cap T \Bigr). \end{equation}

\begin{rqe} Figure \ref{surfaces} represents a schematic version of the intersection $S \cap S'$. Since the surfaces are represented by 1-dimensional segments, the part $(c_{\alpha} \cap c_{\beta}') \times [\frac{1}{3}, \frac{2}{3}]$ is not visible.
\end{rqe}

The following Lemma ensures that if two multicurves are homologous in the surface $\Sigma$ and do not intersect a third multicurve $c$, then, after a slight modification, they remain homologous in $\Sigma\setminus\nu(c)$. This fact will be used later in the proof of Theorem \ref{1-cycle}. 

\begin{lemma}
\label{lemma0}
Let $c_1 = \bigsqcup_{1 \leq i \leq g} n_i \alpha_i$ and~$c_2 = \bigsqcup_{1 \leq i \leq g} m_i \beta_i$ be two multicurves of $\Sigma$. Suppose that $c_1$ and~$c_2$ are homologous in $H_1(\Sigma)$ and that there exists $J\subset\{1,...,g\}$ such that $c_2\cap \alpha_j = \emptyset$ for all $j\in J$. Then there exists $c_1'$, a multicurve formed by parallel copies of~$\alpha_i$, satisfying $[c_1'] = [c_1]$ in $H_1(\Sigma)$ and $[c_1'] = [c_2]$ in $H_1(\Sigma \setminus \bigsqcup_{j \in J} \nu(\alpha_j))$.
\end{lemma}

In particular, the multicurve $c_1'$ can be obtained from $c_1$ by isotopy and addition of pairs $(\alpha_i, -\alpha_i)$.  

\begin{proof}
Up to renumbering the curves $\alpha_i$, we can assume that $J = \{1, \ldots, h\}$ where $h \leq g$. We denote by~$P$~the surface $\Sigma \setminus (\nu(\alpha_1) \cup \ldots \cup \nu(\alpha_h))$. This is a surface of genus $g-h$ with $2h$ boundary components. We denote by~$e_1, e_1', \ldots, e_h, e_h'$ the curves parallel to these boundary components so that the map induced by the inclusion~$P \hookrightarrow \Sigma$~is described by:
\[
\begin{aligned}
\iota_*: H_1(P) & \longrightarrow H_1(\Sigma) \\
[e_i] & \longmapsto [\alpha_i] \\
[e_i'] & \longmapsto -[\alpha_i].
\end{aligned}
\]

Let $c_1''$ be a multicurve in $P$ isotopic to $c_1$ in $\Sigma$. It can be obtained by slightly isotoping the $\alpha_i$ curves on the side of $e_i$ or $e_i'$. Since $c_2$ and $c_1''$ are homologous in $\Sigma$, it exists $\lambda_i\in\Z$ such that
\[ 
[c_2] = [c_1'' + \sum_{i=1}^h \lambda_i(e_i + e_i')] \in H_1(P).
\]
Then we set $c_1':=c_1'' +  \sum_{i=1}^h \lambda_i (e_i+e_i')  \in \Sigma$. We can verify that $[c_1']=[c_2]\in H_1(P)$ and $[c_1']=[c_1]\in H_1(\Sigma)$. 
\end{proof}

In what follows, we use $(\ref{intersection_transverse})$ to express the intersection $S'\cap S$ in terms of arcs extracted from the multicurves~$c_\alpha$ and $c_\beta'$ in the surface $\Sigma$. We begin by considering the intersection between $T'$ and a thickening of $c_\alpha$. 

\begin{lemma}
\label{lemma1}
The intersection $T' \cap (c_{\alpha} \times [0, \frac{1}{3}])$ projected onto $\Sigma\times\{\frac{1}{3}\}$ is given by the~1-chain~$\sum_{k=1}^N \sum_{i=1}^{2n_k} r_i^k a_i^k$. 
\end{lemma}

\begin{rqe}
The result of Lemma \ref{lemma1} is independent of the choice of the surface $T'$. If we take another surface~$U'$ such that $\partial U' = c_{\beta}' - c_{\alpha}'$, then $T'\cup (-U')$ is a closed oriented surface in $\Sigma\times [0,\frac{1}{3}]$ and $(T'\cup~(-U')) \cap~(c_{\alpha} \times [0, \frac{1}{3}])$ is a family of disks in~$(c_{\alpha} \times [0, \frac{1}{3}])$. Thus, the projection on $\Sigma\times\{\frac{1}{3}\}$ is trivial.  
\end{rqe}

\begin{proof}[Proof]
We denote by $R$ and $r$ the maximum and the minimum of the coefficients $r_i^k$ associated with the arcs of~$c_{\alpha}$ : 
\[
\begin{aligned}
R := \max_{\substack{k \in \{1, \ldots, N\} \\ i \in \{1, \ldots, 2n_k\}}} (r_i^k) \\
r := \min_{\substack{k \in \{1, \ldots, N\} \\ i \in \{1, \ldots, 2n_k\}}} (r_i^k). 
\end{aligned}
\]

Let $t_0 = \frac{1}{3} > t_1 > t_2 > \ldots > t_{R-r} > 0$ and let $\rho : c_\alpha \to \mathbb{N}$ be the function defined by 
\[
\rho(x) = \left\{\begin{array}{cl}
r_i^k & \mbox{ if }  x\in \mathrm{Int(c^k_\alpha)} \\
\max(r^k_{i-1},r^k_i) & \mbox{ if } x \in a_{i-1}^k\cap a_i^k.
\end{array}\right.
\]
See Figure \ref{fig:rho} for an example. 
\begin{figure}[t]
\includegraphics[width=0.8\textwidth]{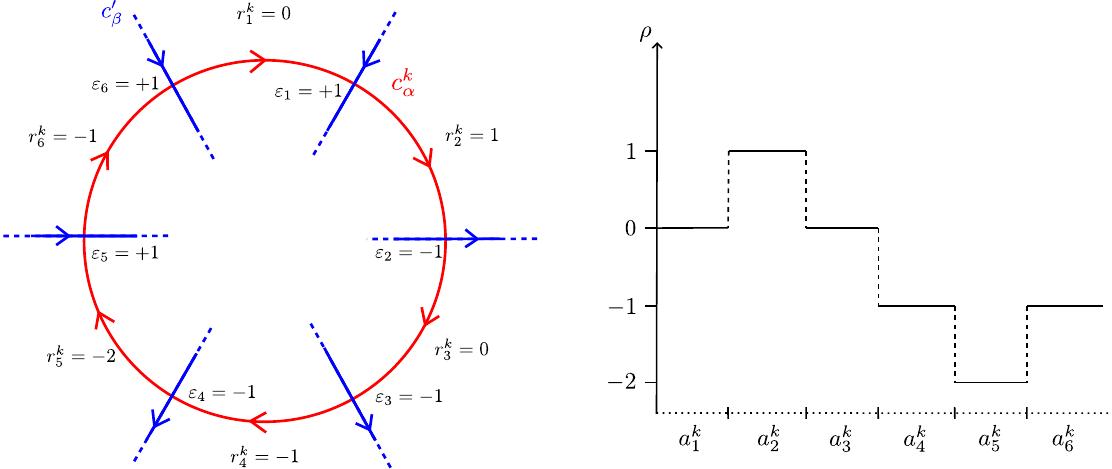}
{\caption{Function $\rho$.}
\label{fig:rho}}
\end{figure}

The idea of the proof is the following. We will construct the surface $T'$ using $R-r$ families of saddles contained in the slices $\Sigma \times [t_j, t_{j+1}]$. Constructing a family of saddles corresponds to truncating the highest steps of the function $\rho$. The construction stops once the function $\rho$ has become constant. 

For this, we introduce $\sigma_t$ the multicurve corresponding to the intersection of $T'$ with the surface $\Sigma \times \{t\}$ :
\[
\sigma_t := T' \cap (\Sigma \times \{t\}). 
\]

In particular, we have $\sigma_{t_0} = c_{\beta}'$. When passing a saddle between $t_j$ and $t_{j+1}$, the curve~$\sigma_t$ is locally modified as shown in Figure \ref{selle}. We say that we perform \textit{a surgery of $\sigma_t$ along the arc $\delta$}.

\vspace{0.2cm}

\begin{figure}[t]
\includegraphics[scale=0.3]{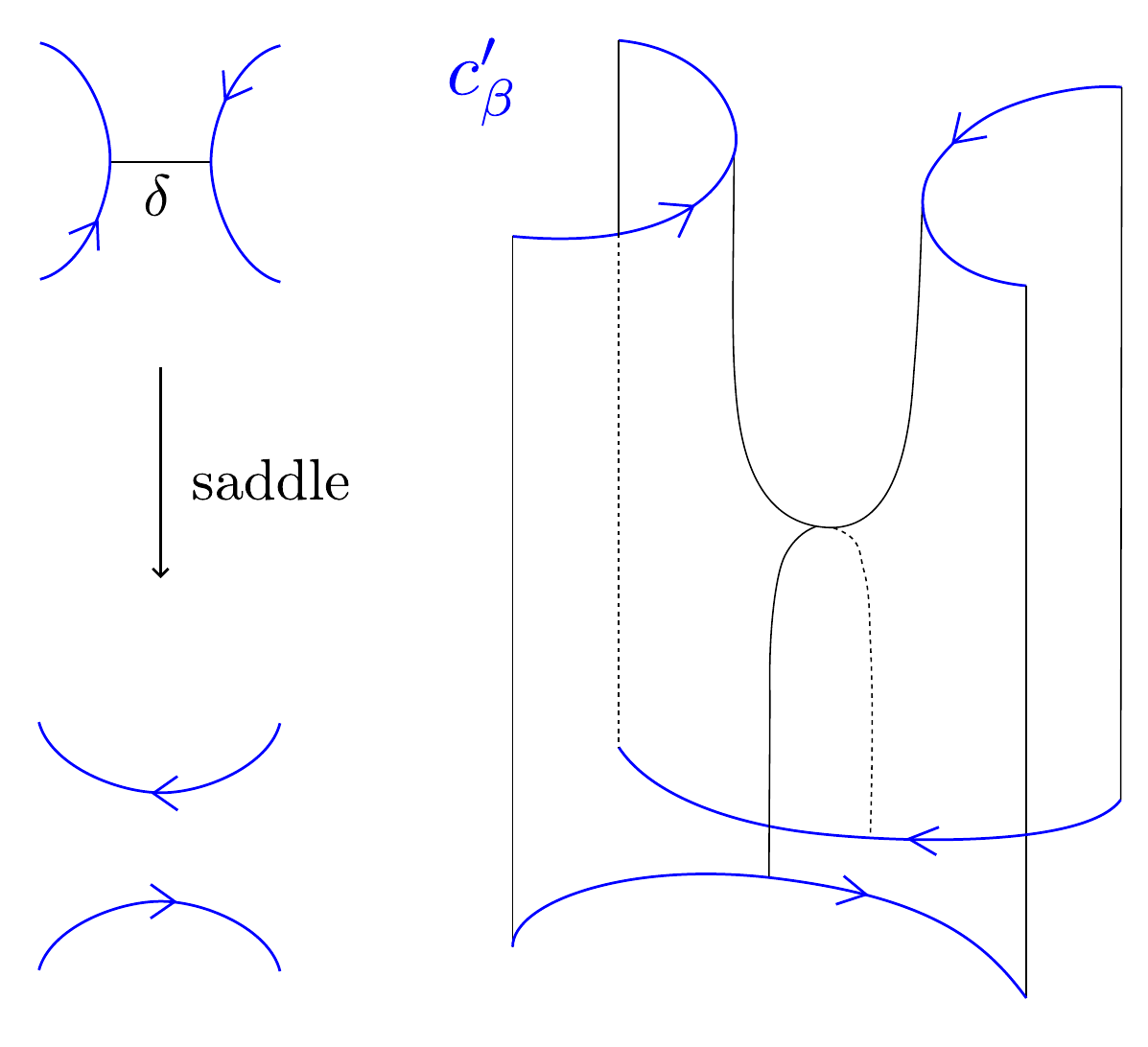}
{\caption{Modification of $c'_\beta$ by surgery along $\delta$.}
\label{selle}}
\end{figure}

We begin by considering the part of $T'$ contained in the slice $\Sigma \times [t_0, t_1]$. We set~$\rho_{t_0}:~=~\rho$ and $\mathcal{A}_{t_0}$ the set of connected components of $\{x \in c_{\alpha} \mid \rho_{t_0}(x) = R\}$. The set $\mathcal{A}_{t_0}$ consists of the arcs $a_i^k$ for which $r_i^k = R$. The first family of saddles in~$\Sigma \times [t_0, t_1]$~is obtained by performing surgeries of the multicurve $\sigma_{t_0}$ along the arcs of $\mathcal{A}_{t_0}$. By definition, the indices $i$ for which $r_i^k$ is maximal satisfy $r_{i-1}^k = r_{i+1}^k = r_i^k - 1$, i.e., $\varepsilon_{i-1} = 1$ and~$\varepsilon_i = -1$. The arcs $a_i^k$ thus connect two portions of $c_{\beta}'$ with opposite orientations, and the surgery along~$a_i^k$ is well-defined. Moreover, these surgeries can be performed simultaneously because the arcs are all disjoint. After surgery, a new multicurve $\sigma_{t_1}$ is obtained. 

Then, we set 
\[
\rho_{t_1} := \rho_{t_0} - \chi_{\mathcal{A}_{t_0}}
\]
where $\chi_{\mathcal{A}_{t_0}}$ is the indicator function of the set ${\mathcal{A}_{t_0}}$. 

The function $\rho_{t_1}$ is still piecewise constant and satisfies $\max(\rho_{t_1}) = R-1$. Similarly, we denote by $\mathcal{A}_{t_1}$ the set of connected components of $\{x \in c_{\alpha} \mid \rho_{t_1}(x) = R-1\}$. This time, the arcs of $\mathcal{A}_{t_1}$ can be concatenations of several arcs $a_i^k$. We add new saddles in the slice $\Sigma \times [t_1, t_2]$ by performing surgeries of $\sigma_{t_1}$ along the arcs of~$\mathcal{A}_{t_1}$, and we obtain a new multicurve $\sigma_{t_2}$.

In general, for $j \in \{1, \ldots, R-r\}$, once the first $j$ families of saddles are constructed in the slices $\Sigma\times~[t_0, t_1], \ldots, \Sigma\times[t_{j-1}, t_j]$, we set
\[
\rho_{t_j} := \rho_{t_{j-1}} - \chi_{\mathcal{A}_{t_{j-1}}}
\]
and
\[
\mathcal{A}_{t_j} := \text{set of connected components of } \{x \in c_{\alpha} \mid \rho_{t_j} = R-j\}.
\]

The part of $T'$ contained in the slice $\Sigma \times [t_j, t_{j+1}]$ is then obtained by surgery of the multicurve $\sigma_{t_j}$ along the arcs of $\mathcal{A}_{t_j}$. Figure \ref{example} illustrates the construction of the saddles in an example.

\begin{figure}[t]
\includegraphics[width=0.7\textwidth]{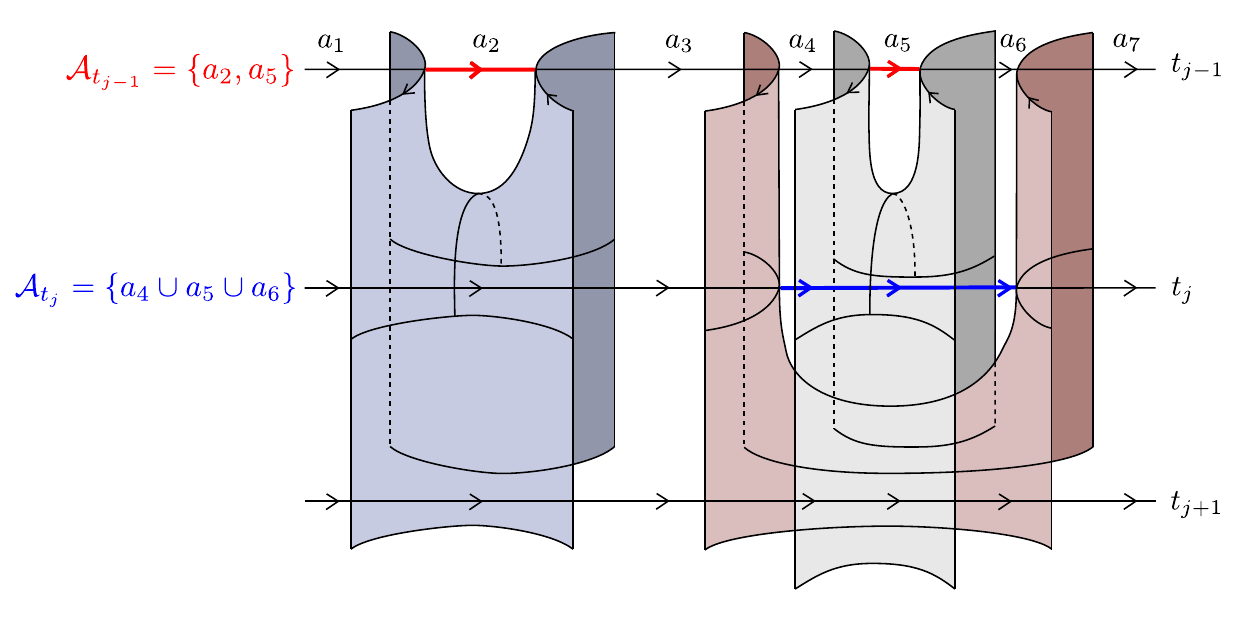}
{\caption{Example of saddle construction between $t_{j-1}$ and $t_{j+1}$.}
\label{example}}
\end{figure}

At the end of all the surgeries, we obtain a multicurve $\sigma_{t_{R-r}}$ on $\Sigma \times \{t_{R-r}\}$. We can verify that
\[
\sigma_{t_{R-r}} \cap c_{\alpha} = \emptyset. 
\]
To see this, consider a point $x \in c_{\alpha} \cap c_{\beta}'$. This point separates two arcs $a_{i-1}^k$ and~$a_i^k$. The largest of the coefficients between $r_{i-1}^k$ and $r_i^k$ is written $R-j$ for some~$j~\in~\{0, \ldots, R-~r\}$. Then $x$ appears as the boundary of one of the arcs of $\mathcal{A}_{t_j}$, and we see that the multicurve $\sigma_{t_{j+1}}$ after surgery along $\mathcal{A}_{t_j}$ no longer intersects $c_{\alpha}$ at $x$. Since the surgeries do not create additional intersections with $c_{\alpha}$, we deduce that $\sigma_{t_{R-r}} \cap c_{\alpha} = \emptyset$.

Moreover, since the surgeries do not modify the homology class of the multicurve, we have
\[
[\sigma_{t_{R-r}}] = [c_{\beta}'] = [c_{\alpha}']. 
\]

By Lemma \ref{lemma0} applied with $c_1 = c_{\alpha}'$ and $c_2 = \sigma_{t_{R-r}}$, it is possible to isotope $c_{\alpha}'$ and add pairs $(\alpha_i, -\alpha_i)$ so as to make it homologous to $\sigma_{t_{R-r}}$ in $\Sigma \setminus \bigsqcup_{i \in I} (\nu(\alpha_i))$, where $I$~is the family of indices of the curves $\alpha$ that appear in $c_{\alpha}$. Since the multicurve $c_{\alpha}'$ is not involved in the construction of the saddles of $T'$ and the calculation of the cup product only depends on its homology class, we can make these modifications to $c_{\alpha}'$. Then, there exists a surface $U' \subset \Sigma \times [0, t_{R-r}]$ satisfying
\[
\partial U' = \sigma_{t_{R-r}} - c_{\alpha}' \text{ and } U' \cap (c_{\alpha} \times [0, t_{R-r}]) = \emptyset.
\]

We then complete the construction of $T'$ by adding $U'$. The surface $T'$ is thus composed of a "saddle" part in $\Sigma \times [t_{R-r}, \frac{1}{3}]$ and a second part in $\Sigma \times [0, t_{R-r}]$ not constructed explicitly.

We finally verify the statement of the lemma for the surface $T'$ we have constructed. Since $U'$ does not intersect $c_{\alpha} \times [0, t_{R-r}]$, it suffices to look at the intersection with the saddles. The saddles intersect the surface~$c_{\alpha} \times [0, t_{R-r}]$ transversely in disjoint arcs connecting two points of $c_{\alpha} \cap c_{\beta}'$ (see Figure \ref{double_selle}).

\begin{figure}[h]
\includegraphics[width=0.75\textwidth]{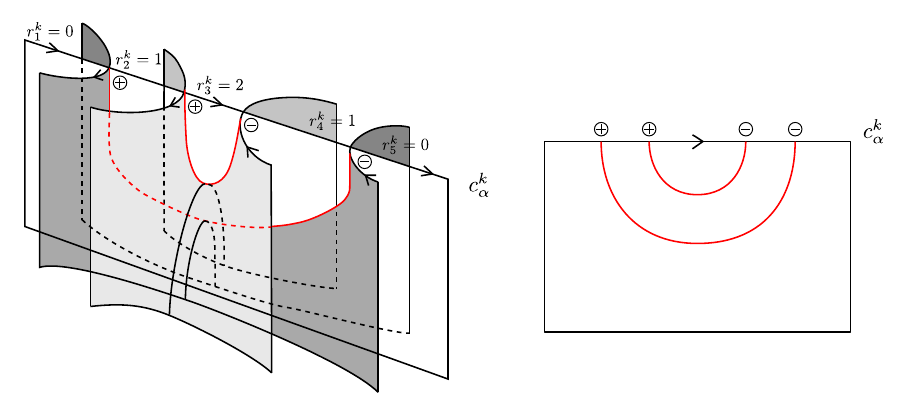}
{\caption{Intersection of $c_\alpha\times[0,t_{R-r}]$ with two nested saddles.}
\label{double_selle}}
\end{figure}

Let $a_i^k \subset c_{\alpha} \times \{\frac{1}{3}\}$. There exists $j \in \{0, \ldots, R-r\}$ such that $r_i^k = R-j$. Then~$a_i^k \subset \mathcal{A}_{t_p}$ for all~$p~\in~\{j, \ldots, R-r\}$. The arc $a_i^k$ is thus above $r_i^k$ saddles of $T'$ (we do not perform surgeries along the arcs of $\mathcal{A}_{t_{R-r}}$). Once projected onto $c_{\alpha} \times \{\frac{1}{3}\}$, the intersection $(c_{\alpha} \times [0, \frac{1}{3}]) \cap T'$ thus corresponds to the 1-chain $\sum_{k=1}^N \sum_{i=1}^{n_k} r_i^k a_i^k$.
\end{proof}

\begin{rqe}
Changing the base point of the numbering of the arcs of $c_\alpha$ results in a vertical shift of the graph of $\rho$. This doesn't modify the way saddles are constructed in the proof of \ref{lemma1}. Thus, the intersection $(c_\alpha \times [0,\tfrac{1}{3}])\cap T'$ is independent of this choice. 
\end{rqe}

We have an analogous lemma for the multicurve $c_{\beta}'$ and the surface $T$:

\begin{lemma}
\label{lemma2}
The intersection $(c_{\beta}' \times [\frac{2}{3}, 1]) \cap T$ is represented by the 1-chain $\sum_{k=1}^M \sum_{i=1}^{2m_k} s_i^k b_i^k$.
\end{lemma}

We note that by construction we have:
\[
\partial \biggl( \sum_{k,i} r_i^k a_i^k \biggr) = c_{\alpha} \cap c_{\beta}' = -\partial \biggl( \sum_{k,i} s_i^k b_i^k \biggr).
\]

The description of the intersection $S \cap S'$ in (\ref{intersection_transverse}) with Lemmas \ref{lemma1} and \ref{lemma2} proves Theorem \ref{1-cycle}.

We can now give a proof of Corollary \ref{triple_cup_product}:

\begin{proof}[Proof of Corollary \ref{triple_cup_product}]

Let $S''$ be the surface in $M = H_{\alpha} \cup (\Sigma \times [0,1]) \cup H_{\beta}$ defined by
\[
S'' = D_{\alpha}'' \cup (c_{\alpha}'' \times [0, \frac{1}{2}]) \cup T'' \cup D_{\beta}''
\]
where $D_{\alpha}'', D_{\beta}''$ are disjoint unions of meridian disks in $H_{\alpha}, H_{\beta}$ and the surface $T''$ is contained in $\Sigma \times [\frac{1}{2}, 1]$. The cup product $x\smile x'$ is represented by the $1$-cycle $\varphi(c_{\alpha}, c_{\beta}')$. We can choose to represent it on the surface~$\Sigma \times \{\frac{1}{4}\}$ (see Figure \ref{S''}). 
\begin{figure}[!ht]
\includegraphics[width=0.7\textwidth]{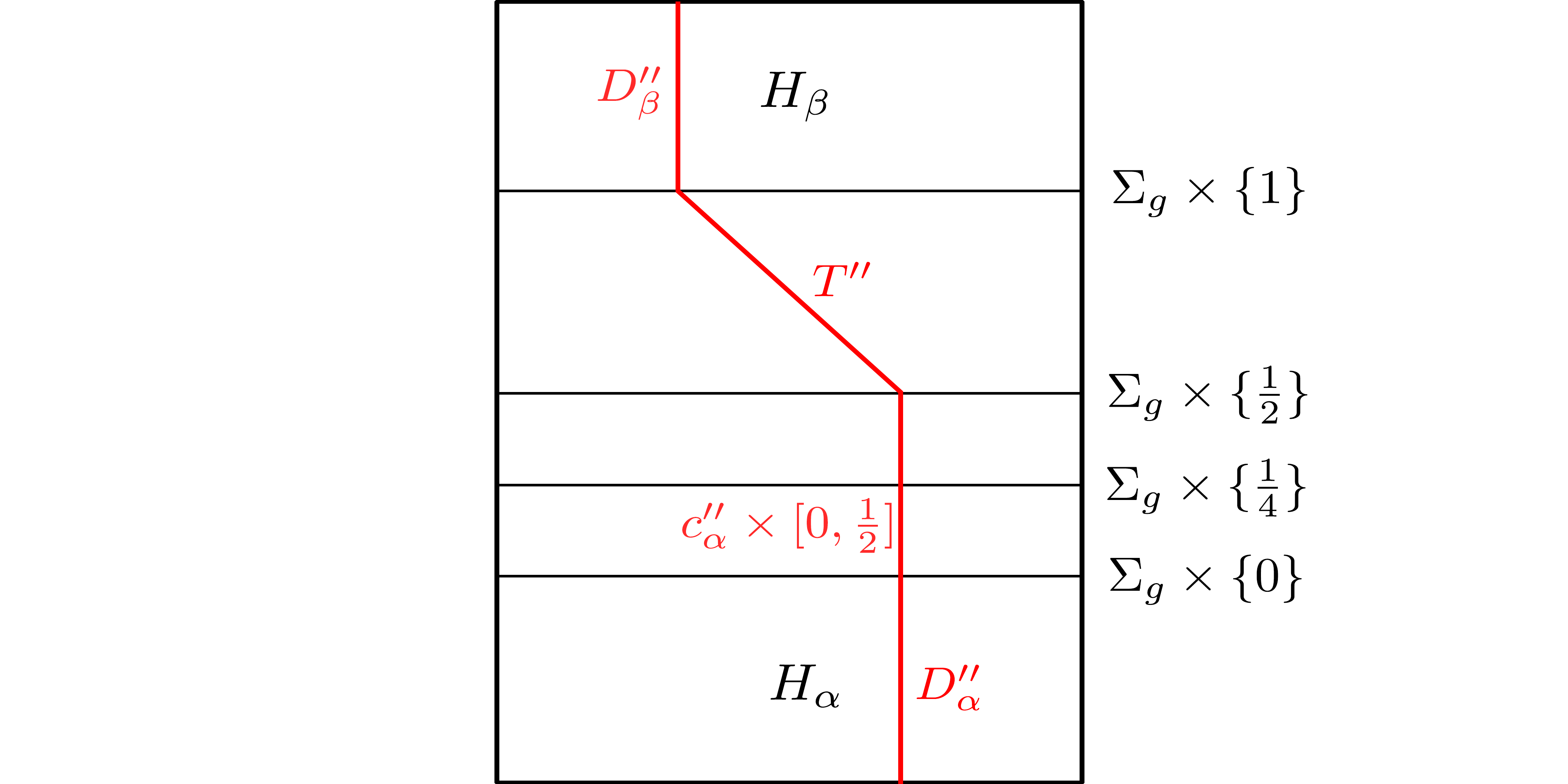}
{\caption{Decomposition of the surface $S''$ in $M$.}
\label{S''}}
\end{figure}
Then, we see that the triple-cup product $\mu(x, x', x'')$ is given by the intersection of $\varphi(c_{\alpha}, c_{\beta}')$ with $c_{\alpha}''$ in the surface $\Sigma$. Let us determine the sign of the intersection. The orientation of $S''$ is determined by that of $T''$, given by $\partial T'' = c_\beta'' - c_\alpha ''$. Figure \ref{fig:signe_intersection} gives a schematic representation of the intersection in~$\Sigma\times\{\frac{1}{4}\}$ (the sign is positive in this example). We then see that the intersection $S''\cap \varphi(c_\alpha,c_\beta')$ coincides with~$\omega([c_\alpha''],\varphi(c_\alpha,c_\beta'))$. 
\begin{figure}[!ht]
\includegraphics[width=0.9\textwidth]{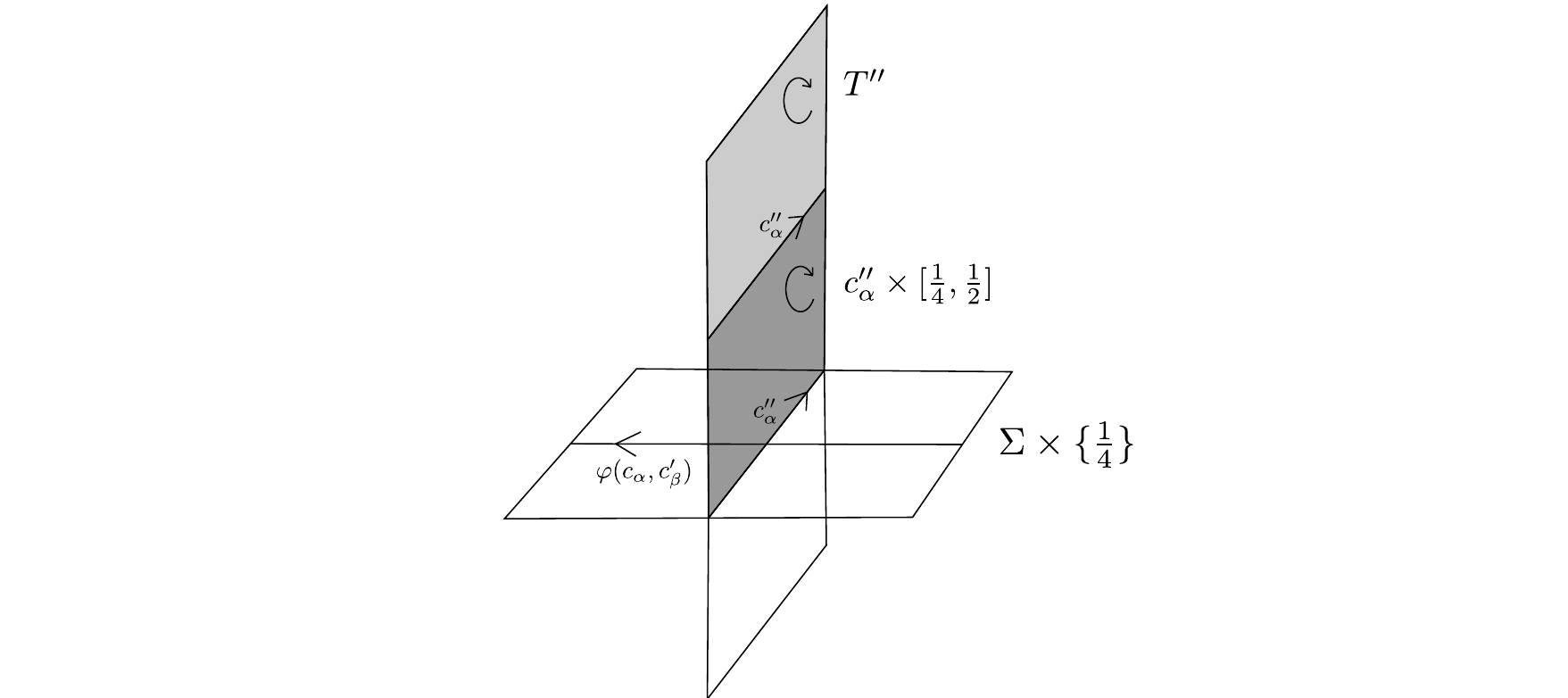}
{\caption{Schematic intersection of the surface $S''$ with $\varphi(c_\alpha,c_\beta')$.}
\label{fig:signe_intersection}}
\end{figure}
\end{proof}

\section{Triple-cup product form and Turaev's homotopy intersection form}
\label{Turaev}

The main result of this section is determining the double cup product form of a $3$-manifold $M=~H_\alpha \cup_\Sigma H_\beta$ in terms of the homotopy intersection form of the Heegaard surface. The proof uses the expression of the double cup product map $\varphi$ of $M$ given by Theorem $\ref{1-cycle}$.

We first review the definition and some useful properties of the homotopy intersection form introduced by Turaev in \cite{Turaev79} (see also \cite{Papa75} and \cite{Perron06}). We refer to \cite{MassuyeauTuraev13} and \cite{MassuyeauTuraev14} for more details. We will use these results to recover some classical properties of the triple-cup product form, such as skew-symmetry and invariance under certain twists along surfaces. 

\subsection{Turaev's homotopy intersection form}

Let $S$ be a compact oriented surface with one boundary component and fix a base point~$*\in \partial S$. Let $\nu$ be the oriented boundary component of $S$ containing the base point~$*$. Let $\bullet$ and $\blacktriangle$ be two other points on $\nu$, "slightly before" and "slightly after" $*$ (see Figure \ref{fig:loops}). For an oriented path $\alpha\in S$ and $p,q$ two distinct points on $\alpha$, we denote by~$\alpha_{pq}$ the path running along $\alpha$ from $p$ to~$q$ in the positive direction. We denote by~$\overline{\alpha}_{qp}$ the same path with opposite orientation. We set $\pi := \pi_1(S,*)$. 

\begin{defn}[Homotopy intersection form]
The homotopy intersection form is the $\Z$-bilinear map $\eta :\Z[\pi]\times\Z[\pi]\to\Z[\pi]$ defined, for any $a,b\in\pi$, by
\[\eta(a,b)  :=  \mathlarger{\sum}_{p\in\alpha\cap\beta} \varepsilon_p(\alpha,\beta) \, \overline{\nu}_{*\bullet}\alpha_{\bullet p}\beta_{p\blacktriangle}\overline{\nu}_{\blacktriangle*}\]
where $\alpha$ is a loop based at $\bullet$ such that $\overline{\nu}_{*\bullet}\alpha\nu_{\bullet*}$ represents $a$, $\beta$ is a loop based at $\blacktriangle$, transverse to $\alpha$, such that~$\nu_{*\blacktriangle}\beta\overline{\nu}_{\blacktriangle*}$ represents $b$, and the sign $\varepsilon_p(\alpha,\beta)=\pm 1$ is equal to~$+1$ if and only if $(T_p\alpha,T_p\beta)$ forms a positive basis of $T_pS$.

\begin{figure}[h]
\includegraphics{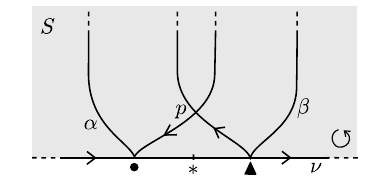}
{\caption{Loops based at $\bullet$ and $\blacktriangle$.}\label{fig:loops}}
\end{figure}

\end{defn}

Denoting by $\varepsilon : \Z[\pi]\to\Z$ the augmentation map which sends an element~$\sum\lambda_i\gamma_i~\in~\Z[\pi]$ to the sum of its coefficients $\sum\lambda_i~\in~\Z$, we define the augmentation ideal $I:=\ker(\varepsilon)$. The ideal $I$ is generated as an abelian group by elements of the form $x-1$ for $x\in \pi$.

\begin{lemma}[{\hspace{-0.5px}\cite[\S 7.2]{MassuyeauTuraev13}}]
\label{Fox}
The intersection form $\eta$ is a Fox pairing, in the sense that it satisfies x
\begin{align}
\eta(xx',y) =  x\,\eta(x',y) + \varepsilon(x')\,\eta(x,y)\label{eq:Fox1} \\
\eta(x,yy') = \eta(x,y)\,y' + \varepsilon(y)\,\eta(x,y') \label{eq:Fox2}
\end{align}
for all $x,x',y,y' \in \Z[\pi]$.
\end{lemma}

These equalities imply the following lemma: 

\begin{lemma}[{\hspace{-0.5px}\cite[\S 7.2]{MassuyeauTuraev13}}]
\label{lemme_ideal}
For $k,l\in \mathbb{N}^*$, we have $\eta(I^k,I^l)\subset I^{k+l-2}$. 
\end{lemma}

We denote by $\pi$ the fundamental group of the surface $S$ and by $(\Gamma_k\pi)_{k\geq 1}$ its lower central series, which is defined by $\Gamma_1\pi := \pi$ and $\Gamma_{k+1}\pi:= [\Gamma_k\pi , \pi]$. Because $\pi$ is a free group, we have the following result due to Magnus \cite{Magnus}:
\begin{lemma}
\label{lemma:Magnus}
For all $k \geq 1, \,\,x-1 \in I^k \Leftrightarrow x\in \Gamma_k\pi$.
\end{lemma}

Finally, Lemmas \ref{Fox}, \ref{lemme_ideal} and \ref{lemma:Magnus} imply the following result: 

\begin{lemma}
\label{lemma:eta_on_Gamma}
For $k,l\in \mathbb{N}^*$, we have $\eta(\Gamma_k\pi,\Gamma_l\pi)\subset I^{k+l-2}$. 
\end{lemma}

\begin{proof}
Let $x\in \Gamma_k\pi$ and $y\in \Gamma_l\pi$. By Lemma \ref{Fox} and the bilinearity of $\eta$, we have $\eta(x,y)=\eta(x-1,y-1)$. Using Lemma \ref{lemma:Magnus}, we obtain that $x-1\in I^k$ and $y-1\in I^l$. We conclude by Lemma \ref{lemme_ideal} that $\eta(x,y)\in I^{k+l-2}$. 
\end{proof}

Let $M$ be a closed, connected, oriented $3$-manifold with a Heegaard splitting $M=~H_\alpha\cup_\Sigma H_\beta$. From now, we work with $S=\Sigma\setminus \mathrm{int}(D)$, where $D$ is a small disk removed from $\Sigma$ to create a boundary component. 
Our goal is to recover the $1$-cycle defined by $\varphi$ from the intersection form $\eta$. We first express $L_\alpha \cap L_\beta$ in terms of quotients of $\pi=\pi_1(S,*)$. Recall that $L_\alpha$ and $L_\beta$ are defined as the kernels of the maps induced in homology by the inclusions of the surface $\Sigma$ into the handlebodies $H_\alpha$ and~$H_\beta$: 
\[
L_\alpha := \ker\bigl(H_1(\Sigma)\to H_1(H_\alpha)\bigr)
\quad \text{and} \quad
L_\beta := \ker\bigl(H_1(\Sigma)\to H_1(H_\beta)\bigr).
\]
We introduce their homotopical counterparts: 
\[
A := \ker(\pi\to\pi_1(H_\alpha))
\quad \text{and} \quad
B := \ker(\pi\to\pi_1(H_\beta)).
\]

\begin{lemma}
We have isomorphisms \[L_\alpha \cap L_\beta \cong \frac{A\cap B[\pi,\pi]}{A\cap[\pi,\pi]}
\cong \frac{B\cap A[\pi,\pi]}{B\cap[\pi,\pi]}.\]
\end{lemma}

\begin{proof}[Proof]
Denote $p:\pi\to\pi/[\pi,\pi]$ the projection onto $H_1(\Sigma)$. Let us check that $p(A)\cap~p(B) = p(A\cap B[\pi,\pi])$. Let~$x\in p(A) \cap p(B)$. There exists $a\in A$ and $b\in B$ such that $x=p(a)=p(b)$. Then $p(ab^{-1})=1$, so~$ab^{-1}\in [\pi,\pi]$ and $a\in B[\pi,\pi]$. We obtain $x\in p(A\cap B[\pi,\pi])$. The other inclusion $p(A\cap B[\pi,\pi])\subset p(A)\cap p(B)$ is also verified because $[\pi,\pi]$ is the kernel of $p$. 

Since the group $A$ is normally generated by $\alpha_1,\alpha_2,...,\alpha_g$ in $\pi$, its image $p(A)$ in the abelianization is generated by the homology class of the $\alpha_i$ curves in $H_1(\Sigma)$, so we have $p(A)=L_\alpha$. Similarly, $p(B)=L_\beta$. Then, we obtain 
\[ 
L_\alpha \cap L_\beta \cong p(A)\cap p(B) \cong p(A\cap B[\pi,\pi]) \cong \frac{A\cap B[\pi,\pi]}{A\cap B[\pi,\pi] \cap [\pi,\pi]} \cong \frac{A\cap B[\pi,\pi]}{A\cap[\pi,\pi]}.
\]
\end{proof}

\begin{lemma}
\label{lemma:eta_AA}
Let $K_\alpha := \mathrm{ker}(\Z[\pi]\to \Z[\pi_1(H_\alpha)])$. We have $\eta(A,A)\subset K_\alpha$.
\end{lemma}

\begin{proof}[Proof]
We denote by $\kappa$ the map $\Z[\pi]\to\Z[\pi_1(H_\alpha)]$, so that $K_\alpha = \mathrm{ker}(\kappa)$. The group~$A$ is normally generated by $\alpha_1,...,\alpha_g$ in $\pi$. Let $i,j\in\{1,...,g \}$ and $x,y\in \pi$. By applying Fox pairing identity ($\ref{eq:Fox1}$) twice, we obtain
\[ \kappa(\eta(x\alpha_i x^{-1},y\alpha_j y^{-1})) = \kappa(x\alpha_i\eta(x^{-1},y\alpha_j y^{-1}) + x\eta(\alpha_i,y\alpha_j y^{-1})+\eta(x,y\alpha_j y^{-1})).\]
Using ($\ref{eq:Fox1}$), we can check $\eta(x^{-1},.)=-x^{-1}\eta(x,.)$, so that 
\begin{align*}
\kappa(\eta(x\alpha_i x^{-1},y\alpha_j y^{-1})) &= \kappa(-x\alpha_i x^{-1}\eta(x,y\alpha_j y^{-1}) + x\eta(\alpha_i,y\alpha_j y^{-1})+\eta(x,y\alpha_j y^{-1}))\\
&=\kappa(x)\kappa(\eta(\alpha_i,y\alpha_j y^{-1})).
\end{align*}
Similarly, it follows from Fox pairing identity ($\ref{eq:Fox2}$) that
\begin{align*}
\kappa(\eta(x\alpha_i x^{-1},y\alpha_j y^{-1})) &= \kappa(x)\kappa(\eta(\alpha_i,y)\alpha_j y^{-1} + \eta(\alpha_i,\alpha_j)y^{-1} + \eta(\alpha_i,y^{-1}) ) \\
&=\kappa(x)\kappa(-\eta(\alpha_i,y^{-1})y\alpha_j y^{-1} + \eta(\alpha_i,\alpha_j)y^{-1} + \eta(\alpha_i,y^{-1}))\\
&=\kappa(x)\kappa(\eta(\alpha_i,\alpha_j))\kappa(y^{-1}).
\end{align*}

For an appropriate numbering and choice of the basing arcs of the $\alpha$-curves, we have:
\[\eta(\alpha_i,\alpha_j) = \left\{
\begin{array}{ccc}
\alpha_i - 1 & \mbox{ if } & i=j \\
-(\alpha_i - 1)(\alpha_j - 1) & \mbox{ if } & i>j \\
0 & \mbox{ if } & i<j.
\end{array}\right.\]
(See, for instance, \cite[(A.5)]{HM2026} for an expression of $\eta$ in an appropriate basis of $\pi$). In every case, we obtain~$\kappa(\eta(\alpha_i,\alpha_j))=0$. Therefore, $\eta(x\alpha_i x^{-1},y\alpha_j y^{-1})\in K_\alpha$ and it follows that $\eta(A,A)\subset K_\alpha$. 
\end{proof}

The following proposition defines the reduction $\overline{\eta}$ of the homotopic intersection form $\eta$ and proves that it is well defined. 

\begin{prop}
Turaev's homotopy intersection form $\eta$ induces a map $\overline{\eta}$ such that the following diagram commutes
\begin{center}
\begin{tikzcd}
(A\cap B[\pi,\pi]) \times (B\cap A[\pi,\pi]) \arrow[d, two heads]\arrow[r,"\eta"] & I \arrow[r] & \displaystyle{\frac{I}{I^2}}\cong H_1(\Sigma) \arrow[r] & \displaystyle{\frac{H_1(\Sigma)}{L_\alpha+L_\beta}}\\
\displaystyle{\frac{A\cap B[\pi,\pi]}{A\cap[\pi,\pi]}} \times  \displaystyle{\frac{B\cap A[\pi,\pi]}{B\cap[\pi,\pi]}}\arrow[rrru,bend right=25,"\overline{\eta}"]
\end{tikzcd}
\end{center}
\end{prop}

\begin{proof}[Proof]

Let $a\in A\cap B[\pi,\pi]$ and $b\in B\cap A[\pi,\pi]$. Since $[a],[b] \in L_\alpha\cap L_\beta$, their homological intersection is trivial and $\varepsilon(\eta(a,b))=0$. Thus $\eta$ is well defined from $(A\cap~B[\pi,\pi])\times(B\cap A[\pi,\pi])$ to $I$. 

The isomorphism $\pi/[\pi,\pi]\cong I/I^2$ is induced by the map $\pi \to I$ which sends $x$ to~$x-1$. We can check that this is indeed an isomorphism using Lemma \ref{lemma:Magnus}, but there is a simpler proof in the case $k=2$. The implication $x\in [\pi,\pi]\Rightarrow x-1 \in I^2$ follows from the identity $[u,v]-1 = ((u-1)(v-1)-(v-1)(u-1))u^{-1}v^{-1}$. For the converse, consider the $\Z$-linear map $h:\Z[\pi]\to \pi/[\pi,\pi], x\mapsto [x]$ and restrict it to $I$. We can check that for~$u,v\in \pi$ we have $h\vert_I((u-1)(v-1)) = h\vert_I((uv-1)-(u-1)-(v-1)) = uvu^{-1}v^{-1}=1$ in $\pi/[\pi,\pi]$. Thus, $h\vert_I$ vanishes on $I^2$.

We denote by $\zeta$ the following composition:
\[
\zeta : \left(A\cap B[\pi,\pi]\right)\times \left(B\cap A[\pi,\pi]\right) \xlongrightarrow[]{\eta} I \longrightarrow I/I^2\cong H_1(\Sigma) \longrightarrow \frac{H_1(\Sigma)}{L_\alpha + L_\beta}.
\]
It remains to check that $\zeta$ descends to the quotients by $A\cap[\pi,\pi]$ and $B\cap[\pi,\pi]$. 

Let $x\in A\cap [\pi,\pi]$, $a\in A$ and $y\in[\pi,\pi]$. Assume that $ay\in B\cap A[\pi,\pi]$. By Lemma~\ref{Fox}, we have
\[\eta(x,ay) = \eta(x,a)y + \eta(x,y). \]
The second term, which evaluates $\eta$ on two elements of $[\pi,\pi]=\Gamma_2\pi$, lies in $I^2$ by Lemma \ref{lemma:eta_on_Gamma}. The first term vanishes in the quotient by $L_\alpha$ by Lemma \ref{lemma:eta_AA}. Thus $\zeta(x,ay)$ is trivial in $H_1(\Sigma)/(L_\alpha + L_\beta)$. Hence,~$\forall\, a'\in~A\cap~[\pi,\pi]$, $w\in A\cap B[\pi,\pi]$, and~$z\in~B\cap~A[\pi,\pi]$, we have
\[
\zeta(a'w,z) = a'\zeta(w,z) + \zeta(a,z) = \zeta(w,z) \in \frac{H_1(\Sigma)}{L_\alpha + L_\beta}.
\]

Therefore, $\zeta$ descends to the quotient
$\frac{A\cap B[\pi,\pi]}{A\cap[\pi,\pi]}$. The same argument applies to
$\frac{B\cap A[\pi,\pi]}{B\cap[\pi,\pi]}$. 
\end{proof}

Let $x \in H^1(M)$ be represented by multicurves $(c_\alpha,c_\beta)$ with $[c_\alpha]~=~[c_\beta]$ in~$L_\alpha \cap~L_\beta$, where $c_\alpha$ (respectively~$c_\beta$) is a disjoint union of $\alpha$ curves (respectively $\beta$ curves) on the surface $\Sigma$. By selecting paths from each component of $c_\alpha$ to the basepoint~$*$, we construct a new multicurve $\tilde{c}_\alpha = \sum_k^N \tilde{c}_\alpha^k$, which can be regarded as an element of $\Z[\pi]$. Similarly, for $x'\in H^1(M)$ represented by $(c_\alpha',c_\beta')$ we construct a multicurve $\tilde{c}_\beta'$ by selecting paths from each component of $c_\beta'$ to the basepoint $*$. 

\begin{lemma}
\label{lemma:paths}
The reduction $\overline{\eta}(\tilde{c}_\alpha,\tilde{c}_\beta')$ is independent of the choice of paths. 
\end{lemma}

\begin{proof}
Since $\eta$ is a bilinear form, we have $\eta(\tilde{c}_{\alpha}, \tilde{c}_\beta') = \sum_k^N \eta(\tilde{c}_\alpha^k,\tilde{c}_\beta')$. Let~$k\in~\{1,...,N\}$. Choosing a different path from $c_\alpha^k$ to $*$ is equivalent to conjugating $\tilde{c}_\alpha^k$ by an element $\gamma\in \pi$. By a straightforward computation using Lemma \ref{Fox}, we get 
\[\eta(\gamma\tilde{c}_{\alpha}^k\gamma^{-1}, \tilde{c}_\beta') - \eta(\tilde{c}_\alpha^k) = (\gamma-1)\eta(\tilde{c}_\alpha^k,\tilde{c}_\beta') - (\gamma \tilde{c}_\alpha^k\gamma^{-1}-1)\eta(\gamma,\tilde{c}_\beta').\]
By taking modulo $I^2$, we obtain 
\[\eta(\gamma\tilde{c}_{\alpha}^k\gamma^{-1}, \tilde{c}_\beta') - \eta(\tilde{c}_\alpha^k) \equiv (\gamma - 1)\omega([{c}_\alpha^k],[{c}_\beta']) -  (\gamma \tilde{c}_\alpha^k\gamma^{-1}-1)\omega([\gamma],[{c}_\beta']) \mbox{ mod } I^2.\]
Since $[c_\beta']\in L_\alpha$, we have $\omega([{c}_\alpha^k],[{c}_\beta']) = 0$. 
The second term vanishes in $H_1(\Sigma)/(L_\alpha+L_\beta)$ because $[\gamma \tilde{c}_\alpha^k \gamma^{-1}] =~[c_\alpha^k]$ is in $L_\alpha$. Thus, we have 
\[ \overline{\eta}(\gamma\tilde{c}_{\alpha}^k\gamma^{-1}, \tilde{c}_\beta') = \overline{\eta}(\tilde{c}_\alpha^k,\tilde{c}_\beta'). \]
Similarly, by applying the same computation on the second argument, we obtain that~$\overline{\eta}(\tilde{c}_\alpha,\tilde{c}_\beta')$ is also independent of the choice of paths from the components of $c_\beta'$ to the basepoint $*$. 
\end{proof}

We will write $\overline{\eta}(c_\alpha,c_\beta')$ instead of $\overline{\eta}(\tilde{c}_\alpha,\tilde{c}_\beta')$ in what follows, for simplicity. We can now prove Theorem \ref{thm:B} stated in the introduction. 

\begin{thm}
\label{thm:eta_varphi}
The following diagram is commutative 

\begin{center}
\begin{tikzcd}
\displaystyle{\frac{A\cap B[\pi,\pi]}{A\cap [\pi,\pi]}} \arrow[d,shift left=0.5cm, "\cong"] \hspace{-1cm}& \times &\hspace{-1cm}  \displaystyle{\frac{B\cap A[\pi,\pi]}{A\cap [\pi,\pi]}}\arrow[d,"\cong",shift left=-0.5cm] \arrow[rr,"\overline{\eta}"]& & \displaystyle{\frac{H_1(\Sigma)}{L_\alpha + L_\beta}}\arrow[dd,"\cong"] \\
L_\alpha \cap L_\beta \arrow[d,"\cong",shift left=0.5cm] \hspace{-1cm}&&\hspace{-1cm} L_\alpha \cap L_\beta\arrow[d,"\cong",shift left=-0.5cm]  &&\\
H^1(M) \hspace{-1cm}& \times &\hspace{-1cm} H^1(M) \arrow[rr,"\smile"] & & H^2(M) 
\end{tikzcd}
\end{center}
where the vertical maps are the isomorphisms induced by the Heegaard splitting of $M$. 
\end{thm}

\begin{proof}[Proof]
Using the notation from the first part, we denote by $c_\alpha^1,\dots,c_\alpha^N$ the $N$ components of $c_\alpha$, and by~$c_\beta'^1,\dots,c_\beta'^M$ the $M$ components of the multicurve $c_\beta$.
For~$k\in~\{1,\dots,N\}$, let $a_1^k,\dots,a_{2n_k}^k$ be the $2n_k$ arcs of $c_\alpha^k\setminus(c_\alpha^k\cap c_\beta')$. Similarly, for~$k\in\{1,\dots,M\}$, let $b_1^k,\dots,b_{2m_k}^k$ denote the $2m_k$ arcs of $c_\beta'^k\setminus(c_\beta'^k\cap c_\alpha)$. 

Since $\overline{\eta}(c_\alpha,c_\beta')$ does not depend on the choice of paths from the multicurves to the base point of $\Sigma$ (see Lemma \ref{lemma:paths}), the arcs that contribute to the reduction $\overline{\eta}(c_\alpha,c_\beta')$ are contained in $c_\alpha$ and $c_\beta'$. Thus, viewed as a $1$-cycle in $H_1(\Sigma)/(L_\alpha+L_\beta)$, we can express $\overline{\eta}(c_\alpha,c_\beta')$ as a linear combination of the arcs $a_i^k$ and $b_i^k$:
\[
\overline{\eta}(c_\alpha,c_\beta') = \sum_{k=1}^N\sum_{i=1}^{2n_k} t_i^k a_i^k + \sum_{k=1}^M\sum_{i=1}^{2m_k} u_i^k b_i^k.
\]

It remains to determine the coefficients $t_i^k$ and $u_i^k$. We first determine the coefficients~$t_i^k$. Number the intersection points $p_i$ of~$c_\alpha^k$ with $c_\beta'$ following the orientation of~$c_\alpha^k$, we have
\[
\eta(c_\alpha^k,c_\beta') = \sum_{j=1}^{2n_k}\varepsilon_{p_j}(c^k_\alpha,c_\beta) \overline{\nu}_{*\bullet}(c_\alpha^k)_{\bullet p_j}(c_\beta')_{p_j\blacktriangle}\overline{\nu}_{\blacktriangle*}.  
\]

For the arc $a_i^k$ to be contained in $(c_\alpha^k)_{\bullet p_j}(c_\beta')_{p_j\blacktriangle}$, we must leave $c_\alpha^k$ after traversing~$a_i^k$, that is, $j\geq i$. It follows that
\[
t_i^k = \sum_{j=i}^{2n_k}\varepsilon_{p_j}(c_\alpha^k,c_\beta') = \sum_{j=i}^{2n_k}-\varepsilon_{p_j}(c_\beta',c_\alpha^k) = r_i^k.
\]

Next, to determine the coefficients $u_i^k$, we number the intersection points of $c_\beta'^k$ with~$c_\alpha$ following the orientation of $c_\beta'^k$. For $b_i^k$ to be contained in $(c_\alpha^k)_{\bullet p_j}(c_\beta')_{p_j\blacktriangle}$, we must arrive on $c_\beta'^k$ before the arc $b_i^k$, hence~$j<i$. Noting that the sign used to define the coefficients $s_i^k$ is the sign of the intersection of $c_\beta^k$ with $c_\alpha$, we obtain
\[
u_i^k = \sum_{j=1}^{i-1}\varepsilon_{p_j}(c_\alpha,c_\beta'^k) = s_i^k.
\]
Thus, $\varphi(c_\alpha,c_\beta') = \overline{\eta}(c_\alpha,c_\beta')$. 
\end{proof}

\subsection{Properties of $\overline{\eta}$}
\label{varphi_properties}

Theorem \ref{thm:eta_varphi} shows that the cup product form can be recovered from $\eta$. The purpose of this section is to verify that classical properties of the triple-cup product form, such as skew-symmetry and invariance under certain twists along surfaces, hold for $\overline{\eta}$. 

We first consider the skew-symmetry of $\overline{\eta}$. In the following, we set \[U:=\frac{A\cap B[\pi,\pi]}{A\cap[\pi,\pi]}\hspace{1cm}V:=\frac{B\cap A[\pi,\pi]}{B\cap[\pi,\pi]}.\]

\begin{prop}
Let $x,x'\in H^1(M)$ and $(c_\alpha, c_\beta),(c_\alpha',c_\beta')$ be pairs of multicurves such that their homology classes satisfy $[c_\alpha]=[c_\beta]$ and $[c_\alpha']=[c_\beta']$ in $L_\alpha\cap L_\beta$ and are associated by duality to $x$ and $x'$. 
The map $\overline{\eta} : U \times V \to H_1(M)$ is skew-symmetric in the sense that $\overline{\eta}(c_\alpha,c'_\beta)=-\overline{\eta}(c'_\alpha,c_\beta)$. 
\end{prop}

\begin{proof}[Proof] 
We observe that $\eta$ is "almost skew-symmetric" in the sense that 
\begin{equation}
\label{eq:eta_skew_sym}
\forall\,a,b\in \pi, \,\eta(a,b)=-aS(\eta(b,a))b - (a-1)(b-1)
\end{equation}
 where $S:\Z[\pi]\to \Z[\pi]$ is the antipode defined by $S(y)=\overline{y}$, where $\overline{y}$ denotes the linear combination obtained from $y$ by reversing the orientation of each loop (\hspace{-0.5px}\cite[\S $7.4$]{MassuyeauTuraev13}). 
 
Denote $\overline{\eta}_{UV}:=\overline{\eta}$ the map from $U\times V$ to $H_1(M)$ and $\overline{\eta}_{VU}$ the same map with exchanged arguments, from~$V\times U$ to $H_1(M)$. The term $(a-1)(b-1)$ in \eqref{eq:eta_skew_sym} lies in~$I^2$. Then, since $[c_\alpha], [c_\beta], [c_\alpha'],[c_\beta']\in L_\alpha\cap L_\beta$, we obtain, after quotienting \eqref{eq:eta_skew_sym} by $I^2$ and $L_\alpha+L_\beta$, 
\[\overline{\eta}_{UV}(c'_\alpha,c_\beta) = \overline{\eta}_{VU}(c_\beta,c'_\alpha). \]

We recall that there is an isomorphism $H_1(\Sigma)\cong I/I^2$ induced by the map $\pi \to I$ which sends $x$ to $x-1$. Then, for every $a,b\in \Z[\pi]$ such that $[a]=[b]=0$ in $H_1(\Sigma)$, we have $a-\varepsilon(a) \in I^2$ and $b-\varepsilon(b) \in I^2$. Using Lemmas \ref{Fox} and \ref{lemme_ideal}, we obtain $\eta(a,b) = \eta(a-\varepsilon(a),b-\varepsilon(b)) \in I^2$, so that $[\eta(a,b)]=0$ in $H_1(\Sigma)$. We deduce that \[0 = [\eta(c_\alpha-c_\beta,c_\beta'-c_\alpha')]=[\eta(c_\alpha,c_\beta')]+[\eta(c_\beta,c_\alpha')] - [\eta(c_\alpha,c_\alpha')] - [\eta(c_\beta, c_\beta')]. \]

We know by Lemma \ref{lemma:eta_AA} that $\eta(A,A)\subset \mathrm{ker}(\Z[\pi]\to\Z[\pi_1(H_\alpha)])$ and $\eta(B,B)\subset~\mathrm{ker}(\Z[\pi]\to\Z[\pi_1(H_\beta)])$. Viewing $c_\alpha,c_\alpha', c_\beta,c_\beta'$ as loops in $\pi$ obtained by concatenation of $\alpha$ and $\beta$-curves, we may regard $c_\alpha,c_\alpha'$ as elements of $A$ and $c_\beta,c_\beta'$ as elements of $B$. Thus $[\eta(c_\alpha,c_\alpha')]\in L_\alpha$ and $[\eta(c_\beta,c_\beta')]\in L_\beta$, and therefore
\[\overline{\eta}_{UV}(c_\alpha,c_\beta')=-\overline{\eta}_{VU}(c_\beta,c_\alpha'). \] 

Finally, we obtain 
\[\overline{\eta}(c_\alpha,c_\beta') = \overline{\eta}_{UV}(c_\alpha,c_\beta') = - \overline{\eta}_{VU}(c_\beta,c_\alpha') = - \overline{\eta}_{UV}(c_\alpha',c_\beta) = - \overline{\eta}(c_\alpha',c_\beta). \]
\end{proof}

For $S$ a compact oriented surface with one boundary component, the mapping class group of $S$ is the group of orientation-preserving diffeomorphisms of $S$ whose restriction to the boundary is the identity, up to isotopy:
\[\M(S) :=  \mbox{Diffeo}^{+,\partial}(S)/\mbox{isotopy rel }\partial.\]
We recall that $\pi$ denotes the fundamental group of the surface $S$ and that $(\Gamma_k\pi)_{k\geq 1}$ is its lower central series, defined by $\Gamma_1\pi := \pi$ and $\Gamma_{k+1}\pi:= [\Gamma_k\pi , \pi]$. For $f$ a diffeomorphism of the surface $S$ whose restriction to $\partial S$ is the identity, we denote $f_\sharp$ the induced homomorphism on $\pi$. We have a group homomorphism $\M(S)\to\mathrm{Aut}(\pi)$ defined by $[f]\mapsto f_\sharp$. Since $\Gamma_{k+1}\pi$ is a characteristic subgroup of $\pi$, that is, a subgroup that is mapped to itself by every automorphism of $\pi$, there is a canonical homomorphism $\mathrm{Aut}(\pi)\to \mathrm{Aut}(\pi/\Gamma_{k+1}\pi)$. 
Then, we denote by $\rho_k$ be the composition \[\rho_k : \M(S) \to \mathrm{Aut}(\pi) \to \mathrm{Aut}(\pi/\Gamma_{k+1}\pi).\] The Torelli group $\I(S)$ is defined as the kernel of $\rho_1$; in other words, it is the subgroup of $\M(S)$ that acts trivially in homology: 
\[\I(S) := \ker(\M(S) \xrightarrow{\rho_1} \mathrm{Aut}(\pi/\Gamma_2\pi)). \]
The Johnson kernel $\K(S)$ is defined as the kernel of $\rho_2$: 
\[\K(S) := \ker(\M(S) \xrightarrow{\rho_2} \mathrm{Aut}(\pi/\Gamma_3\pi)). \]

This subgroup determines a surgery equivalence relation for $3$-manifolds in the following way: 

\begin{defn}[$J_2$-equivalence]
Two closed, connected, oriented $3$-manifolds $M$ and $M'$ are \textit{$J_2$-equivalent} if there exists a Heegaard splitting $H_g\cup_{h} (-H_g)$ of $M$, and~$s\in~\K(\partial H_g\setminus \mathrm{int}(D))$, where $D$ is a disk in $\partial H_g$, such that $M'\cong H_g\cup_{h\circ s} (-H_g)$. 
\end{defn}

\begin{rqe}
This definition is equivalent to twisting any surface $S$, not necessarily a Heegaard surface, in~$M$ by an element of $\K(S)$ (see for example \cite[Proposition $2.10$]{WB}). 
\end{rqe}

The $J_2$-equivalence can be reformulated in terms of Heegaard diagrams. The manifold $M$ is $J_2$-equivalent to $M'$ if there exists $(\Sigma,\alpha,\beta)$ a Heegaard diagram of $M$ and~$s\in~\K(\Sigma\setminus \mathrm{int}(D))$ such that $(\Sigma,s(\alpha),\beta)$ is a Heegaard diagram of $M'$ (or equivalently if there exists $s'\in\K(\Sigma\setminus \mathrm{int}(D)) $ such that $(\Sigma,\alpha,s'(\beta))$ is a Heegaard diagram of $M'$). 

Cochran, Gerges, and Orr proved that the isomorphism class of the triple-cup product form is an invariant of the $J_2$-equivalence of $3$-manifolds \cite[Proposition 3.17]{CochranGergesOrr}. Therefore the expression of $\overline{\eta}$ should be invariant when we apply an element of the Johnson kernel to one of the family of curves $\alpha$ or $\beta$. This is what we verify in the following. 

\begin{prop}
\label{prop:eta_johnson}
Let $D$ be a small disk in $\Sigma$. Let $a,b$ be two multicurves on $\Sigma\setminus \mathrm{int}(D)$ with~$[a],[b]\in~L_\alpha\cap~L_\beta$ and let $\psi\in \K(\Sigma\setminus \mathrm{int}(D))$. We have \[\overline{\eta}(a,b) = \overline{\eta}(\psi(a),b)=\overline{\eta}(a,\psi(b)). \]
\end{prop}

\begin{proof}[Proof]
Since $\psi\in\K(\Sigma\setminus int(D))$, there exists $\gamma\in\Gamma_3\pi$ such that $\psi(a)=a\gamma$. Then, by Lemma \ref{Fox} we obtain 
\[\eta(\psi(a),b)=\eta(a\gamma,b)=a\eta(\gamma,b) + \eta(a,b).\]
Using Lemma \ref{lemma:eta_on_Gamma}, we have $\eta(\Gamma_3\pi,\Gamma_1\pi)\subset I^2$. Hence, $\overline{\eta}(\Gamma_3\pi,.)=0$, and we deduce that~$\overline{\eta}(\psi(a),b)=\overline{\eta}(a,b)$. The same applies for $\overline{\eta}(a,\psi(b))$.
\end{proof}

Thus, the form $\overline{\eta}$ remains unchanged when we apply an element of the Johnson kernel to  one family of curves, but it can be modified when considering an element of the Torelli subgroup. Let $M$ be a closed connected oriented $3$-manifold with a Heegaard splitting~$M=H_g\cup_h (-H_g)$, where $h$ denotes the gluing map. For a small disk~$D\subset \Sigma=\partial H_g$ and $f\in \I(\Sigma\setminus \mathrm{int}(D))$, we denote $M_f = H_g \cup_{h\circ f} (-H_g)$. Since $f$ acts trivially in homology, $M$ and $M_f$ have the same cohomology groups. But their ring structure may be different. As we shall recall, this difference is controlled by the Johnson homomorphism: 
\[\begin{array}{cccccc}
\tau_1 :& \I(\Sigma\setminus \mathring{D}) &\to& \mathrm{Hom}(\pi/\Gamma_2\pi, \Gamma_2\pi/\Gamma_3\pi) & \hookrightarrow & \mathrm{Hom}(I/I^2, I^2/I^3) \\
 & f &\mapsto& \Bigl([x]\mapsto\rho_2(f)([x])\cdot [x]^{-1}\Bigr) & \mapsto & \Bigl([x-1]\mapsto [f(x)\cdot x^{-1}-1]\Bigr).
\end{array}
\]
Specifically, $\Gamma_2\pi/\Gamma_3\pi$ is canonically isomorphic to $\Lambda^2 H_1(\Sigma)$ via the map $[e_1,e_2]\mapsto~e_1\wedge e_2$, so $\tau_1(f)$ defines a map \[\tau_1 : \I(\Sigma\setminus \mathrm{int}(D)) \to \mathrm{Hom}(H_1(\Sigma),\Lambda^2 H_1(\Sigma))\cong(H_1(\Sigma))^* \otimes \Lambda^2 H_1(\Sigma).\] 
Using the intersection pairing $\omega$ of the surface $\Sigma$, we identify $H_1(\Sigma)$ with its dual, via~$x\mapsto \omega(.,x)$, and we convert $\tau_1$ into a map from $\I(\Sigma\setminus \mathring{D})$ to $H_1(\Sigma)\otimes \Lambda^2(H_1(\Sigma))$. There is a natural inclusion of $\Lambda^3H_1(\Sigma)$ in~$H_1(\Sigma)\otimes \Lambda^2H_1(\Sigma)$, given by 
\[\begin{array}{ccc}
\Lambda^3 H_1(\Sigma) & \hookrightarrow & H_1(\Sigma)\otimes \Lambda^2 H_1(\Sigma)\\
a\wedge b \wedge c & \mapsto & a\otimes (b\wedge c) + b\otimes (c\wedge a) + c \otimes (a\wedge b).
\end{array} \]
Johnson proved that the image of $\tau_1$ lies in $\Lambda^3 H_1(\Sigma)$ \cite{JohnsonSurvey}, so $\tau_1$ can be viewed as a homomorphism from~$\I(\Sigma\setminus \mathrm{int}(D))$ to $\Lambda^3 H_1(\Sigma)$.  

Similarly, the triple cup product form $\mu_M$ can be regarded as an element of $\Lambda^3(H_1(M)/(\mathrm{Tors}(H_1(M))$. This follows from the isomorphisms $\Lambda^3(H_1(M)/(\mathrm{Tors}(H_1(M))) \cong \Lambda^3(H^1(M)^*) \cong (\Lambda^3 H^1(M))^*$, where the last isomorphism comes from the non-degeneracy of the bilinear form 
\[ \begin{array}{ccc}
\Lambda^3 H^1(M) \times \Lambda^3(H^1(M)^*) &\to & \Z \\
(x_1,x_2,x_3),(y_1,y_2,y_3)&\mapsto& det((y_i(x_j))_{1\leq i,j\leq 3}).
\end{array}\]

Recall that, since $f$ is an element of the Torelli subgroup, we have $H_1(M_f)\cong H_1(M)$. Then, $\mu_{M_f}$ can also be regarded as an element of $\Lambda^3 (H_1(M)/\mathrm{Tors}(H_1(M)))$ and it makes sense to consider the difference~$\mu_{M_f}-\mu_M$. The following is a generalization of Proposition \ref{prop:eta_johnson}:

\begin{prop}
\label{prop:mu_and_tau_1}
Let $i$ be the inclusion of the Heegaard surface $\Sigma$ in $M$. Then,  
\[ \mu_{M_f} - \mu_M = (\Lambda^3 i_*)(\tau_1(f))  \]
where $\mu_{M_f}$ and $\mu_{M}$ are regarded as elements of $\Lambda^3 (H_1(M)/\mathrm{Tors}(H_1(M)))$ and $\tau_1(f)$ is viewed as an element of $\Lambda^3 H_1(\Sigma)$.
\end{prop}

The relationship between $\mu_{M_f}$, $\mu_M$ and $\tau_1(f)$ was originally established by Johnson \cite{JohnsonSurvey}. Here, we present an alternative proof using the connection between the triple-cup product form $\mu$ and the form $\eta$. Before getting into the proof, we need a short computational lemma. 
\begin{lemma}
\label{lemma:eta_I2}
Let $x,y,z\in\pi$. We have \[\eta(x,[y,z])=\omega(x,y)(z-1)-\omega(x,z)(y-1) \,\,\mathrm{mod}\,\, I^2.\] 
\end{lemma}

\begin{proof}
A straightforward computation using Lemma \ref{Fox} gives
\[ \eta(x,[y,z]) = (\eta(x,y)(z-1)-\eta(x,y)(y-1))y^{-1}z^{-1}.\] 

Since $\omega(x,y) = \eta(x,y)$ mod $I$ and $y^{-1}z^{-1} = 1$ mod $I$, it follows that
\begin{align*}
\eta(x,[y,z]) &= (\omega(x,y)(z-1)- \omega(x,z)(y-1))y^{-1}z^{-1} \mbox{ mod } I^2 \\
&= \omega(x,y)(z-1)- \omega(x,z)(y-1) \mbox{ mod } I^2.
\end{align*}
\end{proof}

\begin{proof}[Proof of Proposition \ref{prop:mu_and_tau_1}]
Let $x,x',x''\in H^1(M)$ represented by pairs of multicurves $(c_\alpha,c_\beta)$, $(c_\alpha',c_\beta')$ and~$(c_\alpha'',c_\beta'')$ in $L_\alpha\cap L_\beta$, respectively. We compute the difference
\begin{align*}
\overline{\eta}(c_\alpha,f(c_\beta')) - \overline{\eta}(c_\alpha,c'_\beta) &= \overline{\eta}(c_\alpha,f(c_\beta')-c_\beta') \\
&= \overline{\eta}(c_\alpha,(f(c_\beta')c_\beta'^{-1}-1)c_\beta') \\
&= \overline{\eta}(c_\alpha,\tau_1(f)(c_\beta')).
\end{align*}
The last equality follows from Lemma \ref{Fox} and the facts that $c_\beta'\in L_\beta$ and $\tau_1(c_\beta')\in~I^2\subset~I$. 

For simplicity, assume that $\tau_1(f)=u \wedge v \wedge w\in \Lambda^3H_1(\Sigma)$ (in general, $\tau_1(f)$ is a sum of such terms). Then, the image of $\tau_1(f)(c_\beta')$ in $\Lambda^2 H_1(\Sigma)$ is given by:
\[\tau_1(f)(c_\beta') = \omega(c_\beta',u)(v\wedge w) + \omega(c_\beta',v)(w\wedge u) + \omega(c_\beta',w)(u\wedge v). \]

Via the isomorphism $\Lambda^2 H_1(\Sigma)\cong[\pi,\pi]/\Gamma_3\pi$, this corresponds to 
\[\tau_1(f)(c_\beta') = \omega(c_\beta',u)[v, w] + \omega(c_\beta',v)[w, u] + \omega(c_\beta',w)[u, v]. \]

Applying Lemma \ref{lemma:eta_I2} with the isomorphism $I/I^2\to H_1(\Sigma)$, $x-1\mapsto x$, we obtain: 
\begin{align}
\label{eq:7}
\begin{split}
\overline{\eta}(c_\alpha,\tau_1(c_\beta'))=&
\hspace{3px}\omega(c_\beta',u)\Bigl(\omega(c_\alpha,v)w-\omega(c_\alpha,w)v\Bigr) \\
&+ \omega(c_\beta',v)\Bigl(\omega(c_\alpha,w)u-\omega(c_\alpha,u)w\Bigr) \\
&+ \omega(c_\beta',w)\Bigl(\omega(c_\alpha,u)v-\omega(c_\alpha,v)u\Bigr).
\end{split}
 \end{align}

Then, by Corollary $\ref{triple_cup_product}$ and Theorem $\ref{thm:eta_varphi}$, the difference $\mu_{M_f}-\mu_M$ is given by:
\begin{align*}
\mu_{M_f}(x,x',x'') - \mu_M(x,x',x'') &= \omega(c_\alpha'',\overline{\eta}(c_\alpha,f(c_\beta')) - \overline{\eta}(c_\alpha,c'_\beta)) \\
&= \omega(c_\alpha'',\overline{\eta}(c_\alpha,\tau_1(f)(c_\beta'))).
\end{align*}

By (\ref{eq:7}), this evaluates as the following determinant:
\[\mu_{M_f}(x,x',x'') - \mu_M(x,x',x'')
=\begin{vmatrix} \omega(c_\alpha,u) & \omega(c_\beta',u) & \omega(c_\alpha'',u) \\
\omega(c_\alpha,v) & \omega(c_\beta',v) & \omega(c_\alpha'',v) \\
\omega(c_\alpha,w) & \omega(c_\beta',w) & \omega(c_\alpha'',w)  \end{vmatrix}.\]

Thus, $\mu_{M_f}-\mu_M$ corresponds to $(\Lambda^3i_*)(u \wedge v \wedge w) = (\Lambda^3i_*)(\tau_1(f))$ in~$\Lambda^3(H_1(M)/\mathrm{Tors}(H_1(M)))$. 
\end{proof}

\section{Examples}
\label{section:examples}

\subsection{The manifold $\Sr^1\times\Sr^1\times\Sr^1$}

The $3$-manifold $\Sr^1\times\Sr^1\times\Sr^1$ can be viewed as a cube with opposite faces identified. We can obtain a Heegaard splitting as follows. The first handlebody consists of a $0$-~handle located the center of the cube, to which three $1$-handles are attached, each passing through a pair of opposite faces of the cube. We can check that the complement is also a handlebody of genus $3$. The corresponding Heegaard diagram is shown in Figure \ref{fig:heegaard_diagram_S1S1S1}. The Heegaard surface (in grey) is of genus $3$. The first two handles are represented by squares $1$ and $2$, identified by horizontal and vertical symmetry. The third handle is represented by inner and outer orange squares, which are identified via the identity map.

\begin{figure}[ht]
\includegraphics[width=0.7\textwidth]{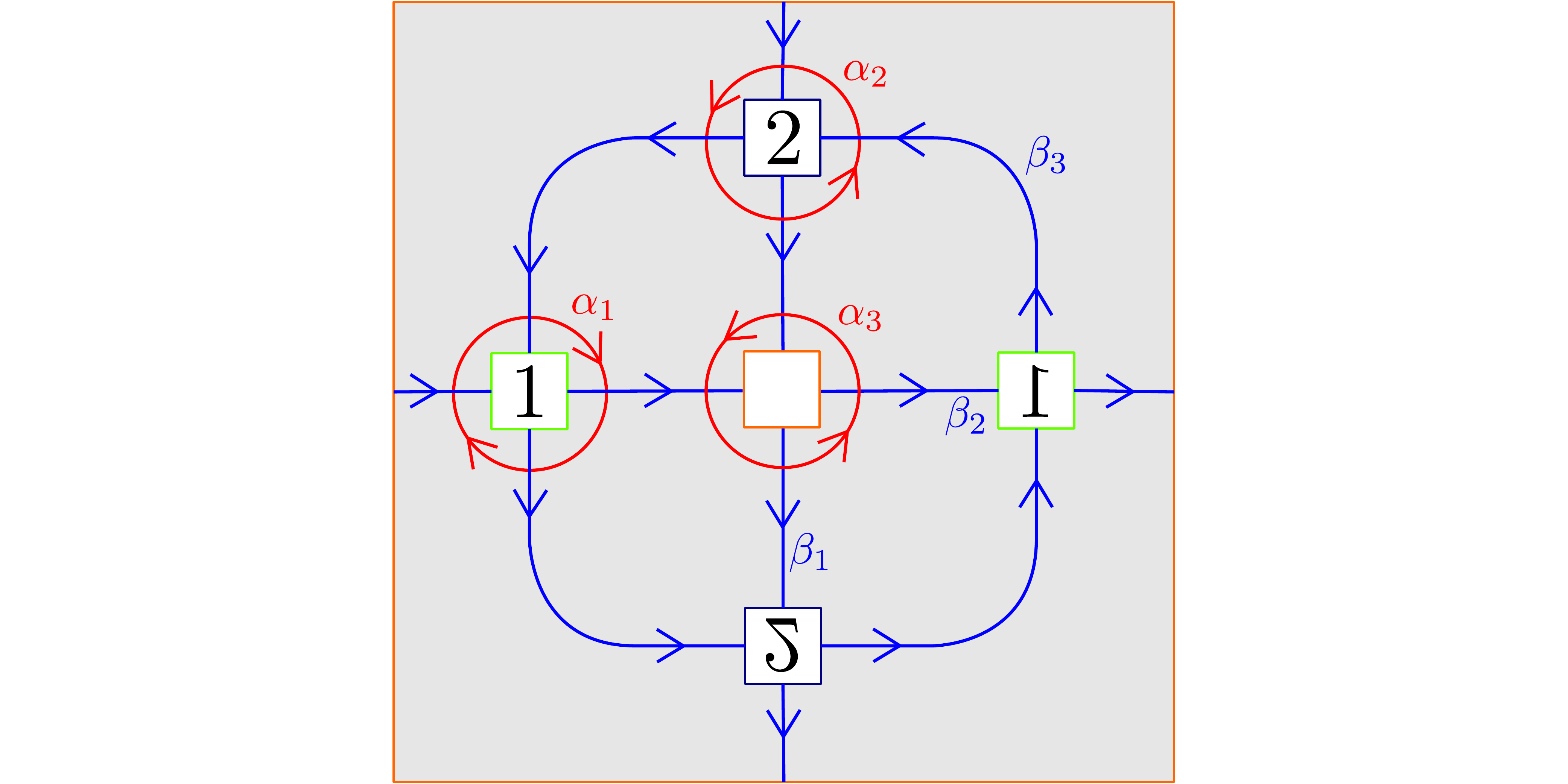}
{\caption{Heegaard diagram of $\Sr^1\times\Sr^1\times\Sr^1$.}
\label{fig:heegaard_diagram_S1S1S1}}
\end{figure}

For $i\in\{1,2,3\}$ we have $[\alpha_i]=[\beta_i]$, so that $L_\alpha=L_\beta$ and \[L_\alpha\cap L_\beta=\langle \alpha_1,\alpha_2,\alpha_3\rangle = \langle \beta_1,\beta_2,\beta_3\rangle.\] Let $x,x',x''$ be a basis of $H^1(\Sr^1\times\Sr^1\times\Sr^1)$, respectively associated by duality with the pairs of multicurves $(\alpha_1,\beta_1), (\alpha_2,\beta_2)$, $(\alpha_3,\beta_3)$ (see Figure \ref{fig:heegaard_diagram_S1S1S1_ca_cb}). Let us compute~$\mu(x,x',x'')$ with this diagram. Applying the formula of Theorem~\ref{1-cycle} we find that~$\varphi(c_\alpha,c_\beta')\in H_1(\Sr^1\times\Sr^1\times\Sr^1)$ is represented by the $1$-cycle shown in Figure~\ref{fig:varphi_in_S1S1S1}. This curve is interseced by $c_\alpha''$ transversely in a single point with positive sign. Therefore, by Corollary \ref{triple_cup_product}, we obtain $\mu(x,x',x'')=1$. 

\begin{figure}[t]
\includegraphics[width=0.7\textwidth]{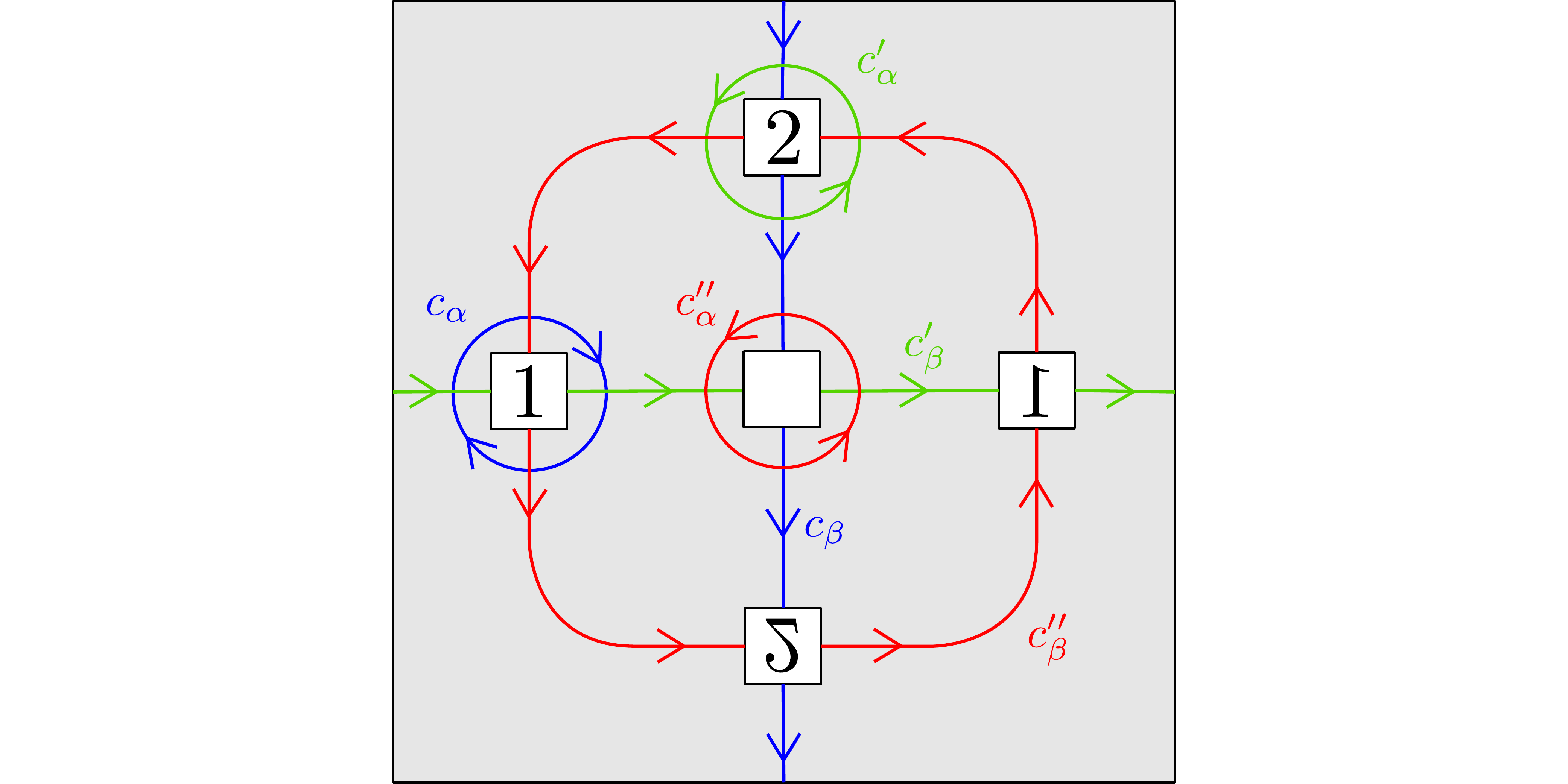}
{\caption{Pairs of multicurves associated with $x$,$x'$ and $x''$.}
\label{fig:heegaard_diagram_S1S1S1_ca_cb}}
\end{figure}

\begin{figure}[!htb]
\includegraphics[width=0.7\textwidth]{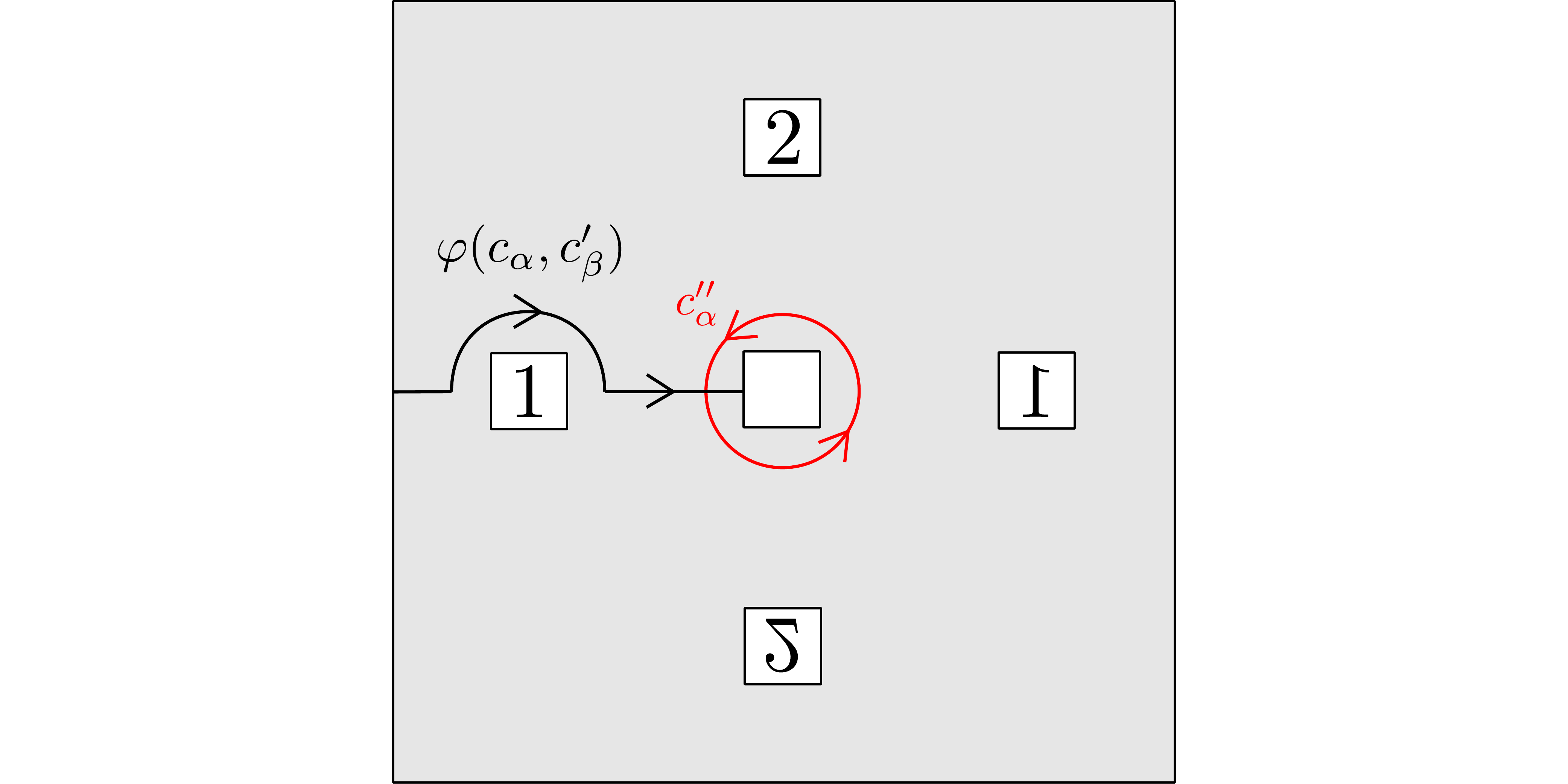}
{\caption{$\varphi(c_\alpha,c_\beta')$ and $c_\alpha''$ curves.}
\label{fig:varphi_in_S1S1S1}}
\end{figure}

\subsection{The manifold $\Sigma_g\times\Sr^1$}

Using the classical representation of $\Sigma_g$ as a $4g$-gon with identified edges, we can view~$\Sigma_g\times\Sr^1$ as a thickened $4g$-gon with top and bottom faces identified via the identity map and vertical faces identified in the same way as in the construction of $\Sigma_g$. We obtain a Heegaard splitting of genus $2g+1$ by cutting out a $0$-handle at the center of the thickened polygon, together with $2g$ horizontal $1$-handles passing through pairs of identified faces, and one vertical $1$-handle passing through the top and bottom faces. Figure~\ref{fig:heegaard_construction_sigma2_S1} and Figure~\ref{fig:heegaard_sigma2_S1} show the construction of the Heegaard splitting and the corresponding Heegaard diagram when~$g=2$. The first four handles are depicted by circles identified under symmetry operations. The last handle is formed by the inner and outer octagons, identified via the identity map. Each $\alpha$-curve is a meridian of the Heegaard surface. 

\begin{figure}[!htb]
\includegraphics[width=0.8\textwidth]{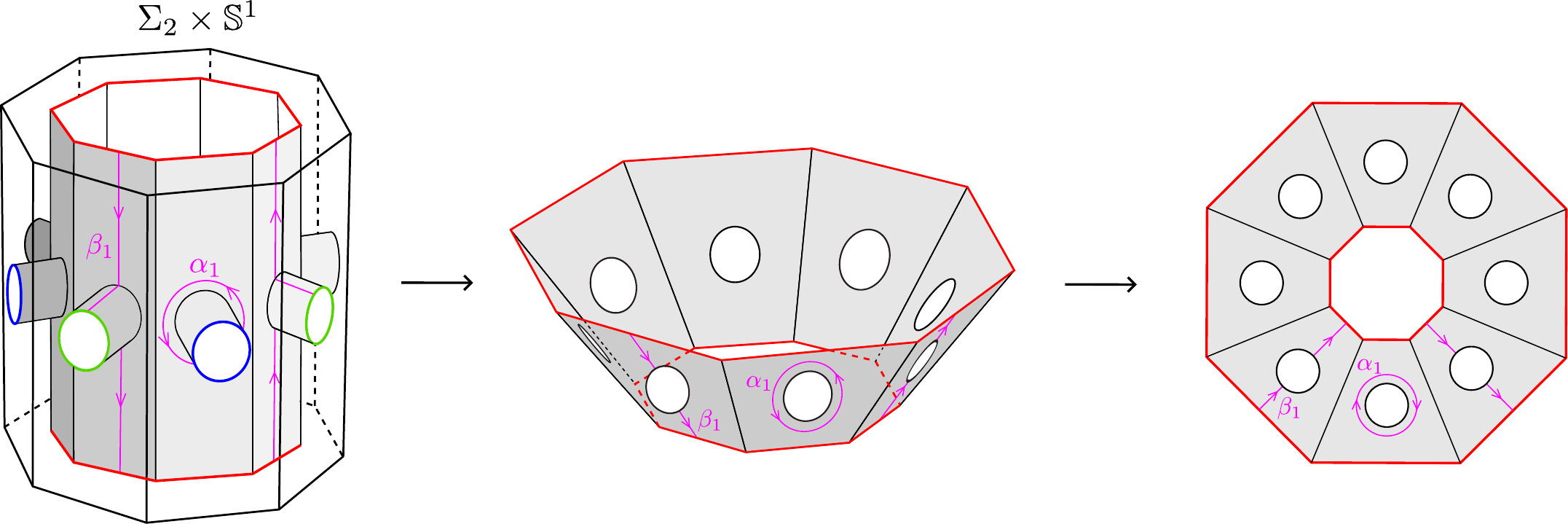}
{\caption{Heegaard surface inside $\Sigma_2\times \Sr^1$.}
\label{fig:heegaard_construction_sigma2_S1}}
\end{figure}

\begin{figure}[!htb]
\begin{center}
\includegraphics[width=0.8\textwidth]{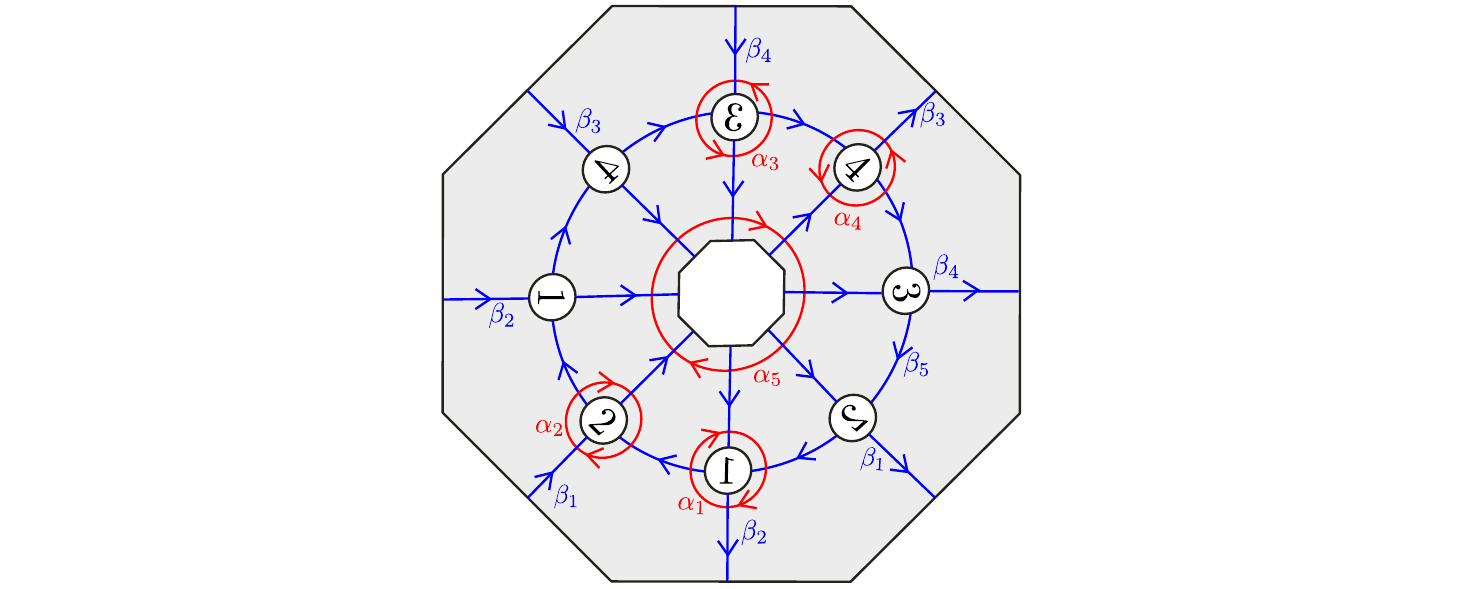}
{\caption{Heegaard diagram of $\Sigma_2\times\Sr^1$.}
\label{fig:heegaard_sigma2_S1}}
\end{center}
\end{figure}

For $i\in\{1,...,2g+1\}$ we have $[\alpha_i]=[\beta_i]$, so that $L_\alpha=L_\beta$ and \[L_\alpha\cap L_\beta=\langle \alpha_1,...,\alpha_{2g+1}\rangle = \langle \beta_1,...,\beta_{2g+1}\rangle.\] 

Let $x_i\in H^1(\Sigma_2\times\Sr^1)$ be the element associated by duality with the pair of curves $(\alpha_i,\beta_i)$. The family $(x_1,...,x_{2g+1})$ forms a basis of $H^1(\Sigma_2\times\Sr^1)$. 

We observe that for $i\neq j$, if $\alpha_i$ does not intersect $\beta_j$, then $\varphi(\alpha_i,\beta_j)=0$ and consequently $\mu(x_i,x_j,.)=0$. 

To illustrate the case where the intersection between $\alpha_i$ and $\beta_j$ is nonempty, we take~$g=2$ and we set $i=1$ and $j=2$. A $1$-cycle representing $\varphi(\alpha_1,\beta_2)$ is shown in Figure \ref{fig:varphi_sigma2_S1}. By Corollary \ref{triple_cup_product} we obtain 
\begin{gather*}\mu(x_1,x_2,x_5)=1\\ \mu(x_1,x_2,x_3)=\mu(x_1,x_2,x_4)=0.\end{gather*}

\begin{figure}[t]
\includegraphics[width=0.8\textwidth]{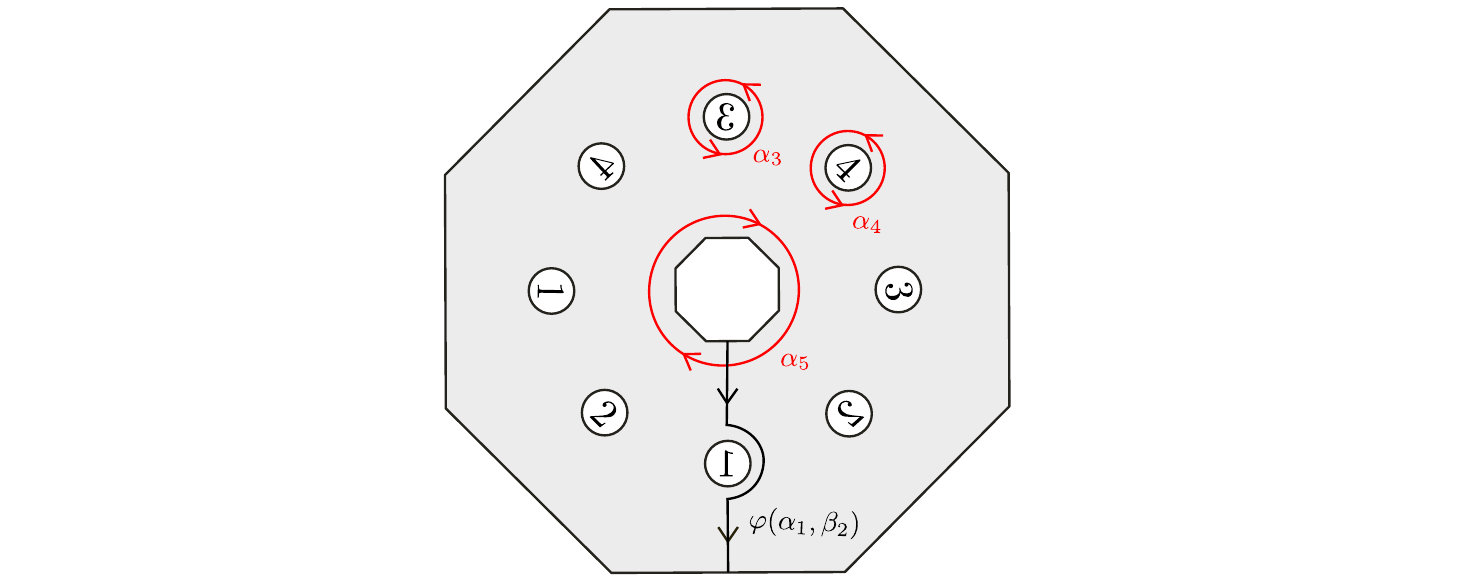}
{\caption{$\varphi(\alpha_1,\beta_2)$.}
\label{fig:varphi_sigma2_S1}}
\end{figure}

In general, for any $g$, one can verify by standard cohomological techniques that the triple-cup product form of $\Sigma_g\times \Sr^1$ is induced by the homological intersection form $\omega$ of $\Sigma_g$. Let $p^*:H^1(\Sigma_g\times \Sr^1)\to H^1(\Sigma_g)$ and $q^*:H^1(\Sigma_g\times \Sr^1)\to H^1(\Sr^1)$ be the maps induced by the inclusions $\Sigma_g\hookrightarrow \Sigma_g\times \Sr^1$ and $\Sr^1\hookrightarrow \Sigma_g\times \Sr^1$. By the Künneth formula, we have an isomorphism
\[\begin{array}{rcc}
H^1(\Sigma_g\times \Sr^1) &\xrightarrow[]{\cong}& (H^1(\Sigma_g)\otimes H^0(\Sr^1))\oplus (H^0(\Sigma)\otimes H^1(\Sr^1)) \\
x\hspace{17px} &\mapsto & (x_{\Sigma}\times 1) + (1\times x_{S})
\end{array}\]
where $x_\Sigma = p^*(x)$ and $x_S = q^*(x)$. 

Then, for $x,y,z\in H^1(\Sigma_g\times \Sr^1)$, we have
\begin{align*}
&\mu(x,y,z) =\,\langle x \smile y \smile z, [\Sigma_g \times \Sr^1]\rangle \\
= \,&\langle ((x_\Sigma \times 1)+(1\times x_S))\smile((y_\Sigma\times 1)+(1\times y_S))\smile((z_\Sigma\times 1)+(1\times z_S)) , [\Sigma_g\times \Sr^1]       \rangle \\
= \,& \langle (x_\Sigma\smile y_\Sigma \smile 1)\times(1\smile 1 \smile z_S),[\Sigma_g] \times [\Sr^1] \rangle \\
&+ \langle (x_\Sigma \smile 1 \smile z_\Sigma)\times(1\smile y_S \smile 1), [\Sigma_g]\times [\Sr^1] \rangle \\
&+ \langle (1\smile y_\Sigma \smile z_\Sigma)\times (x_S\smile 1 \smile 1), [\Sigma_g]\times[\Sr^1]\rangle. 
\end{align*}

The last equality follows from the formula $(a\times b)\smile(c\times d)=(a\smile c)\times(b\smile~d)$, for all $a,c\in H^*(\Sigma_g)$ and $b,d\in H^*(\Sr^1)$ (note that $x_\Sigma\smile y_\Sigma \smile z_\Sigma =0$ and $x_S \smile~y_S =~x_S\smile z_S = y_S \smile z_S =0$), and the identity $[\Sigma_g\times \Sr^1]=[\Sigma_g]\times[\Sr^1]$.

 Then, we obtain: 
\begin{align*}
\mu(x,y,z)=\,& \langle p^*(x\smile y),[\Sigma_g]\rangle . \langle z_S,[\Sr^1]\rangle \\&- \langle p^*(x\smile z),[\Sigma_g]\rangle . \langle y_S,[\Sr^1]\rangle \\&+ \langle p^*(y\smile z),[\Sigma_g]\rangle . \langle x_S,[\Sr^1]\rangle\\
=&\,\tilde{\omega}(x_\Sigma, y_\Sigma).\langle z_S,[\Sr^1]\rangle - \tilde{\omega}(x_\Sigma,z_\Sigma).\langle y_S,[\Sr^1]\rangle + \tilde{\omega}(y_\Sigma,z_\Sigma).\langle x_S,[\Sr^1]\rangle.
\end{align*} 
where $\tilde{\omega} : H^1(\Sigma_g)\times H^1(\Sigma_g)\to\Z$ is given by $\tilde{\omega}(.,.) = \omega(P(.),P(.))$, where $P:~H^1(\Sigma_g)\to H_1(\Sigma_g)$ denotes the Poincaré isomorphism and $\omega :~H_1(\Sigma_g)\times~H_1(\Sigma_g)\to\Z$ is the intersection form on $\Sigma_g$. 

We can recover this formula using the Heegaard diagram described above. For~$i\in~\{1,...,2g\}$, the pair~$(\alpha_i,\beta_i)$ represents the homology class of the surface $D_{\alpha_i}\cup S_i \cup D_{\beta_i}$ in $H_2(\Sigma_g\times \Sr^1)$, where $D_{\alpha_i},D_{\beta_i}$ are disks in $H_\alpha$, $H_\beta$, respectively, and $S_i\subset\Sigma$ is a surface such that $\partial S = \alpha_i \cup \beta_i = \partial D_{\alpha_i}\cup \partial D_{\beta_i}$. The dual $1$-cocycle, which corresponds to $x_i$ in $H^1(\Sigma_g\times\Sr^1)$, is represented by the "horizontal" curve which intersects transversely the surface once. Similarly, the pair $(\alpha_{2g+1}, \beta_{2g+1})$ corresponds to the "horizontal" surface in~$\Sigma_{2g}\times\Sr^1$ and the dual $1$-cocycle is represented by the "vertical" curve corresponding to the $\Sr^1$ factor (see Figure \ref{fig:surfaces_octogone}). Consequently, we have $q^*(x_i)=0$ for $i\in\{1,...,2g\}$ and $p^*(x_{2g+1})=0$. 

\begin{figure}[t]
\includegraphics[width=0.5\textwidth]{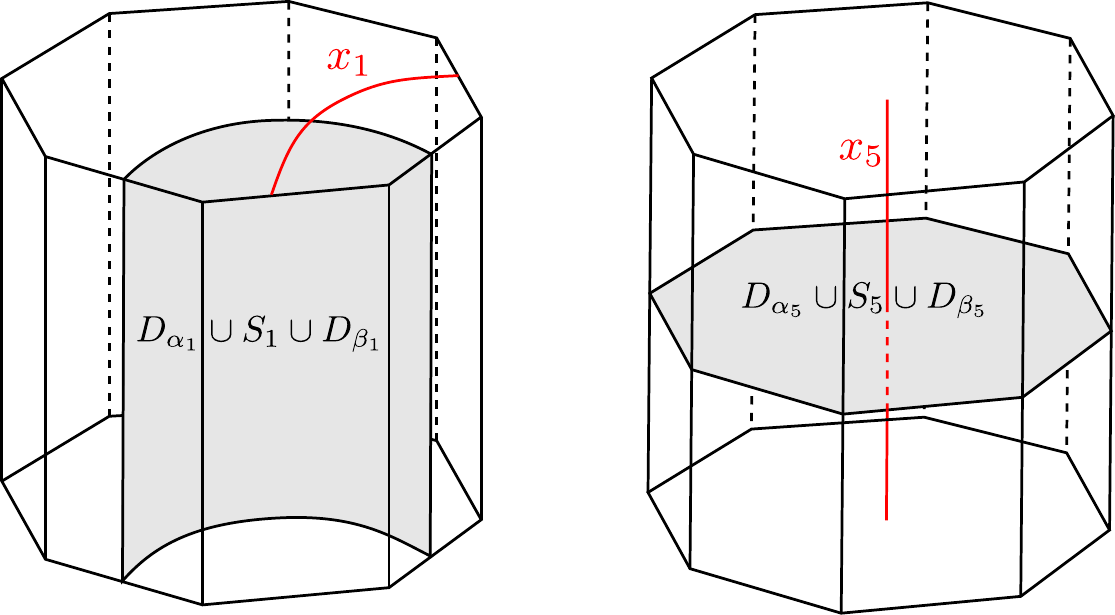}
{\caption{Poincaré duals of $D_{\alpha_1}\cup S_1\cup D_{\beta_1}$ and $D_{\alpha_5}\cup S_5\cup D_{\beta_5}$ in $H^1(\Sigma_2\times \Sr^1)$.}
\label{fig:surfaces_octogone}}
\end{figure}

It follows that the isomorphism $H^1(\Sigma_{2g}\times\Sr^1)\cong H^1(\Sigma_{2g}) \times H^1(\Sr^1)$ sends $x_1,...,x_{2g}$ to a symplectic basis of $H^1(\Sigma_{2g})$ and $x_{2g+1}$ to a generator of $H^1(\Sr^1)\cong \Z$. Therefore, for $i,j,k\in\{1,...,2g \}$ the triple-cup product form of $\Sigma_{2g}\times\Sr^1$ is given by 
\begin{gather*}
\mu(x_i,x_j,x_{2g+1})=\omega(P(p^*(x_i),P(p^*(y_j)))\\
\mu(x_i,x_j,x_k)=0.
\end{gather*}

\bibliographystyle{alpha}
\bibliography{triple_cup_product_heegaard_diagrams_biblio}

\end{document}